\documentclass{amsart}
\pdfoutput=1

\usepackage[utf8]{inputenc}
\usepackage[T1]{fontenc}
\usepackage{lmodern}
\usepackage[english]{babel}
\usepackage{amsmath,amssymb,amsthm}
\usepackage[alphabetic]{amsrefs}
\usepackage[marginratio={1:1,1:1},totalwidth=400pt,totalheight=600pt]{geometry}
\usepackage{microtype}
\usepackage{hyperref}
\usepackage[shortlabels]{enumitem}
\usepackage{url}
\usepackage{color}
\usepackage[all]{xy}
\usepackage{graphicx}

\theoremstyle{plain}
\newtheorem{thm}{Theorem}
\newtheorem{lem}[thm]{Lemma}
\newtheorem{prop}[thm]{Proposition}
\newtheorem{cor}[thm]{Corollary}

\theoremstyle{definition}
\newtheorem{defn}[thm]{Definition}

\newtheorem{rem}[thm]{Remark} 
\newtheorem{conj}[thm]{Conjecture}

\newcommand{\C}{\mathbb{C}}
\newcommand{\N}{\mathbb{N}}
\newcommand{\Z}{\mathbb{Z}}
\newcommand{\Q}{\mathbb{Q}}
\newcommand{\R}{\mathbb{R}}
\newcommand{\CA}{\mathcal{A}}
\newcommand{\CO}{\mathcal{O}}
\newcommand{\QQ}{\mathcal{Q}}
\newcommand{\BB}{\mathcal{B}}
\newcommand{\DD}{\mathcal{D}}
\newcommand{\F}{\mathcal{F}}
\newcommand{\K}{\mathcal{K}}
\newcommand{\M}{\mathcal{M}}

\newcommand{\U}{\mathcal U}
\newcommand{\op}{\operatorname}
\newcommand{\Aut}{\op{Aut}}
\newcommand{\id}{\op{id}}
\newcommand*\onto{\ensuremath{\joinrel\relbar\joinrel\twoheadrightarrow}} 
\newcommand*\into{\ensuremath{\lhook\joinrel\relbar\joinrel\rightarrow}}  

\newcommand{\PS}{P}

\numberwithin{thm}{section}

\title[On the $K$-theory of $C^*$-algebras arising from integral dynamics]{On the $K$-theory of $C^*$-algebras arising from \protect\\ integral dynamics}

\author{Sel\c{c}uk Barlak}
\address{Department of Mathematics and Computer Science\\
University of Southern Denmark\\
Campusvej 55\\
DK-5230 Odense M\\
Denmark}
\email{barlak@imada.sdu.dk}

\author{Tron Omland}
\address{Department of Mathematics\\University of Oslo\\P.O.Box 1053 Blindern\\NO-0316 Oslo\\Norway}
\email{trono@math.uio.no}

\author{Nicolai Stammeier}
\address{Department of Mathematics\\University of Oslo\\P.O.Box 1053 Blindern\\NO-0316 Oslo\\Norway}
\email{nicolsta@math.uio.no}

\date{June 8, 2016}

\begin{document}
\begin{abstract}
We investigate the $K$-theory of unital UCT Kirchberg algebras $\QQ_S$ arising from families $S$ of relatively prime numbers. It is shown that $K_*(\QQ_S)$ is the direct sum of a free abelian group and a torsion group, each of which is realized by another distinct $C^*$-algebra naturally associated to $S$. The $C^*$-algebra representing the torsion part is identified with a natural subalgebra $\CA_S$ of $\QQ_S$. For the $K$-theory of $\QQ_S$, the cardinality of $S$ determines the free part and is also relevant for the torsion part, for which the greatest common divisor $g_S$ of $\{p-1 : p \in S\}$ plays a central role as well. In the case where $\lvert S \rvert \leq 2$ or $g_S=1$ we obtain a complete classification for $\QQ_S$. Our results support the conjecture that $\CA_S$ coincides with $\otimes_{p \in S} \CO_p$. This would lead to a complete classification of $\QQ_S$, and is related to a conjecture about $k$-graphs.
\end{abstract}
\maketitle

\section{Introduction}
Suppose $S$ is a non-empty family of relatively prime natural numbers and consider the submonoid of $\N^\times$ generated by $S$. Its action on $\Z$ by multiplication can be represented on $\ell^2(\Z)$ by the bilateral shift $U$ and isometries $(S_p)_{p \in S}$ defined by $U\xi_n = \xi_{n+1}$ and $S_p\xi_n = \xi_{pn}$. The associated $C^*$-algebra $C^*\bigl(U,(S_p)_{p \in S}\bigr)$ admits a universal model $\QQ_S$ that is generated by a unitary $u$ and isometries $(s_p)_{p \in S}$, subject to\vspace*{-2mm}
\[\begin{array}{c} s_ps_q = s_qs_p \text{ for } q  \in S,\quad s_pu=u^ps_p, \quad \text{ and } \quad \sum\limits_{m=0}^{p-1}u^ms_p^{\phantom{*}}s_p^*u^{-m} = 1. \end{array}\vspace*{-2mm}\] 
By results of \cites{KOQ} or \cite{Sta1}, $\QQ_S$ is isomorphic to $C^*\bigl(U,(S_p)_{p \in S}\bigr)$ and belongs to the class of unital UCT Kirchberg algebras. In view of the Kirchberg-Phillips classification theorem \cites{Kir,Phi}, the information on $S$ encoded in $\QQ_S$ can therefore be read off from its $K$-theory.

In special cases, $\QQ_S$ and its $K$-theory have been considered before: If $S$ is the set of all primes, then $\QQ_S$ coincides with the algebra $\QQ_\N$ from \cite{CuntzQ} and it follows that $K_i(\QQ_S) = \Z^\infty$ for $i=0,1$ and $[1]=0$. The other extreme case, where $S=\{p\}$ for some $p \geq 2$, appeared already in \cite{Hir}: Hirshberg showed that $(K_0(\QQ_{\{p\}}),[1],K_1(\QQ_{\{p\}})) = (\Z \oplus \Z/(p-1)\Z,(0,1),\Z)$. This result was recovered later in \cite{KatsuraIV} and \cites{CuntzVershik} as a byproduct. Note that $\QQ_{\{p\}}$ coincides with Katsura's algebra $\CO(E_{p,1})$, see \cite{KatsuraIV}*{Example~A.6}. Moreover, Larsen and Li analyzed the situation for $p=2$ in great detail, see \cite{LarsenLi}. The similarities and differences among these known cases raise several questions: 
\begin{enumerate}[(i)]
\item Is $K_1(\QQ_S)$ always torsion free?
\item Is $2 \in S$ the only obstruction to torsion in $K_0(\QQ_S)$? 
\item What is the $K$-theory of $\QQ_S$ in the general case of $\lvert S \rvert \geq 2$?
\item What does $\QQ_S \cong \QQ_T$ reveal about the relationship between $S$ and $T$?  
\end{enumerate}
Through the present work, we provide a complete description in the case of $\lvert S \rvert=2$, for which the $K$-theory of $\QQ_S$ satisfies 
\[(K_0(\QQ_S),[1],K_1(\QQ_S)) = (\Z^2 \oplus \Z/g_S\Z,(0,1),\Z^2 \oplus \Z/g_S\Z),\] 
where $g_S=\gcd(\{p-1 : p \in S\})$, see Theorem~\ref{thm:main result}~(c). Thus we see that the first two questions from above have a negative answer (for instance, consider $S=\{3,5\}$ and $S=\{5,6\}$, respectively). More generally, we completely determine $K_*(\QQ_S)$ in the case of $\lvert S \rvert \leq 2$ or $g_S = 1$, see Theorem~\ref{thm:main result}, and conclude that $\QQ_S \cong \QQ_T$ if and only if $\lvert S \rvert=\lvert T \rvert$ and $g_S=g_T$ in this case. In addition, Theorem~\ref{thm:main result} substantially reduces the problem in the remaining case of $\lvert S \rvert \geq 3$ and $g_S >1$. Thereby we also make progress towards a general answer to the remaining questions (iii) and (iv) from above.

In order to prove Theorem~\ref{thm:main result}, we first compare the stabilization of $\QQ_S$ to the $C^*$-algebra $C_0(\R) \rtimes N \rtimes H$, where $N=\Z\bigl[\{\frac{1}{p} : p \in S\}\bigr]$, $H$ is the subgroup of $\Q_+^\times$ generated by $S$, and the action comes from the natural $ax+b$-action of $N \rtimes H$ on $\R$, see Section~\ref{sec:comp with real dyn}. This approach is inspired by methods of Cuntz and Li from \cite{CLintegral2}. However, the final part of their strategy is to use the Pimsner-Voiculescu sequence iteratively, see \cite{CLintegral2}*{Remark~3.16}, and depends on having free abelian $K$-groups, which does not work in our situation. Instead, we show that $K_*(\QQ_S)$ decomposes as a direct sum of a free abelian group and a torsion group, both arising in a natural way from two distinguished $C^*$-algebras related to $\QQ_S$, see Theorem~\ref{thm:decomposition of K-theory} and Corollary~\ref{cor:torsion and free part K-theory}. The determination of the torsion free part of $K_*(\QQ_S)$ uses a homotopy argument, and thereby benefits heavily from the comparison with real dynamics. This allows us to prove that the rank of the torsion free subgroup of $K_i(\QQ_S)$ equals $2^{\lvert S \rvert-1}$ for both $i=0,1$, see Proposition~\ref{prop:K-theory torsion free part}. 

The torsion subgroup of $K_*(\QQ_S)$ is realized by the semigroup crossed product $M_{d^\infty} \rtimes^e_\alpha H^+$, where $d$ is the product of all primes dividing some element of $S$, $H^+$ is the submonoid of $\N^\times$ generated by $S$, and the action $\alpha$ is inherited from a semigroup crossed product description of $\QQ_S$, see Corollary~\ref{cor:torsion and free part K-theory}. Appealing to the recently introduced machinery for equivariantly sequentially split $*$-homomorphisms from \cite{BarSza1}, we show that $M_{d^\infty} \rtimes^e_\alpha H^+$ is a unital UCT Kirchberg algebra, just like $\QQ_S$, see Corollary~\ref{cor:UHF into BD seq split cr pr}. Quite intriguingly, this paves the way to identify $M_{d^\infty} \rtimes^e_\alpha H^+$ with the subalgebra $\CA_S = C^*(\{u^ms_p : p \in S, 0 \leq m \leq p-1\})$ of $\QQ_S$, see Corollary~\ref{cor:subalgebra for torsion part}. That is why we decided to name $\CA_S$ the \emph{torsion subalgebra}. This $C^*$-algebra is interesting in its own right as, for instance, it admits a model as the boundary quotient $\QQ(U)$ of a particular right LCM submonoid $U$ of $\N \rtimes H^+$, see Proposition~\ref{prop:A_S as BQ of U}. As explained in Remark~\ref{rem:A_S as BQ of U}, this gives rise to a remarkable diagram for the semigroup $C^*$-algebras and boundary quotients related to the inclusion of right LCM semigroups $U \subset \N \rtimes H^+$. 

With regards to the $K$-theory of $\CA_S$ and hence $\QQ_S$, the $k$-graph description for finite $S$ obtained in Corollary~\ref{cor:tor subalgebra via k-graphs} is more illuminating: The canonical $k$-graph for $\CA_S$ has the same skeleton as the standard $k$-graph for $\bigotimes_{p \in S}\CO_p$, but uses different factorization rules, see Remark~\ref{rem:Lambda_S flip}. It is apparent from the given presentation that $\CA_S$ is isomorphic to $\CO_p$ for $S=\{p\}$. If $S$ consists of two relatively prime numbers $p$ and $q$, then a result from \cite{Evans} shows that $\CA_S$ coincides with $\CO_p \otimes \CO_q$. For the remaining cases, we extract vital information on $\CA_S$ by applying Kasparov's spectral sequence \cite{kasparov} (see also \cite{barlak15}) to the $H^+$-action $\alpha$ on $M_{d^{\infty}}$, see Theorem~\ref{thm:K-theory for A_S}. More precisely, we obtain that $\CA_S$ is isomorphic to $\bigotimes_{p \in S}\CO_p$ if $\lvert S \rvert \leq 2$ or $g_S=1$. In the latter case, it actually coincides with $\CO_2$. Additionally, we show that the order of every element in $K_*(\CA_S)$ divides $g_S^{2^{\lvert S \rvert - 2}}$. As we remark at the end of this work, the same results can be obtained by employing the $k$-graph representation of $\CA_S$ and using Evans' spectral sequence \cite{Evans} for the $K$-theory of $k$-graph $C^*$-algebras. In view of these results, it is very plausible that $\CA_S$ always coincides with $\bigotimes_{p \in S}\CO_p$. This would be in accordance with Conjecture~\ref{conj:k-graph}, which addresses independence of $K$-theory from the factorization rules for $k$-graphs under certain constraints. If $\CA_S \cong \bigotimes_{p \in S}\CO_p$ holds for all $S$, then we get a complete classification for $\QQ_S$ with the rule that $\QQ_S$ and $\QQ_T$ are isomorphic if and only if $\lvert S \rvert = \lvert T \rvert$ and $g_S=g_T$, see Conjecture~\ref{conj:K-theory of QQ_S}.

At a later stage, the authors learned that Li and Norling obtained interesting results for the multiplicative boundary quotient for $\N \rtimes H^+$ by using completely different methods, see \cite{LN2}*{Subsection~6.5}. Briefly speaking, the multiplicative boundary quotient related to $\QQ_S$ is obtained by replacing the unitary $u$ by an isometry $v$, see Subsection~\ref{subsec:BQ for ADS} for details. As a consequence, the $K$-theory of the multiplicative boundary quotient does not feature a non-trivial free part. It seems that $\CA_S$ is the key to reveal a deeper connection between the $K$-theoretical structure of these two $C^*$-algebras. As this is beyond the scope of the present work, we only note that
the inclusion map from $\CA_S$ into $\QQ_S$ factors through the multiplicative boundary quotient as an embedding of $\CA_S$ and the natural quotient map. The results of \cite{LN2} together with our findings indicate that this embedding might be an isomorphism in $K$-theory. This idea is explored further in \cite{Sta3}*{Section~5}.

The paper is organized as follows: In Section~\ref{sec:prelim}, we set up the relevant notation and list some useful known results in Subsection~\ref{subsec:notation and basics}. We then link $\QQ_S$ to boundary quotients of right LCM semigroups, see Subsection~\ref{subsec:BQ for ADS}, and $a$-adic algebras, see Subsection~\ref{subsec:a-adic algs}. These parts explain the central motivation behind our interest in the $K$-theory of $\QQ_S$. In addition, the connection to $a$-adic algebras allows us to apply a duality theorem from \cite{KOQ}, see Theorem~\ref{thm:duality}, making it possible to invoke real dynamics. This leads to a decomposition result for $K_*(\QQ_S)$ presented in Section~\ref{sec:K-theory}, which essentially reduces the problem to determining the $K$-theory of $\CA_S$. The structure of the torsion subalgebra $\CA_S$ is discussed in Section~\ref{sec:torsion part}. Finally, the progress on the classification of $\QQ_S$ we obtain via a spectral sequence argument for the $K$-theory of $\CA_S$ is presented in Section~\ref{sec:classification}.

\subsection*{Acknowledgments}
The first named author was supported by SFB~$878$ \emph{Groups, Geometry and Actions}, GIF Grant~$1137$-$30.6/2011$, ERC through AdG~$267079$, and the Villum Fonden project grant `Local and global structures of groups and their algebras' (2014-2018). The second named author was supported by RCN through FRIPRO~$240913$. A significant part of the work was done during the second named author's research stay at Arizona State University, and he is especially grateful to the analysis group at ASU for their hospitality. He would also like to thank the other two authors of this paper for their hospitality during two visits to M\"{u}nster. The third named author was supported by ERC through AdG~$267079$ and by RCN through FRIPRO~$240362$. We are grateful to Alex Kumjian for valuable suggestions.

\section{Preliminaries}\label{sec:prelim}
\subsection{Notation and basics}\label{subsec:notation and basics}
Throughout this paper, we assume that $S \subset \N^\times \setminus \{1\}$ is a non-empty family of  relatively prime numbers. We write $p|q$ if $q \in p\N^\times$ for $p,q \in \N^\times$. Given $S$, we let $\PS := \{p \in \N^\times : p \text{ prime and } p|q \text{ for some } q \in S\}$. Also, we define $d:= \prod_{p \in \PS}p$ (which is a supernatural number in case $S$ is infinite, see Remark~\ref{rem:supernatural}) and $g_S$ to be the greatest common divisor of $\{p-1 : p \in S\}$, i.e.\ $g_S := \gcd(\{p-1 : p \in S\})$.

Recall that $\N^\times$ is an Ore semigroup with enveloping group $\Q_+^\times$, that is, $\N^\times$ embeds into $\Q_+^\times$ (in the natural way) so that each element $q \in \Q_+^\times$ can be displayed as $p^{-1}q$ with $p,q \in \N^\times$. The subgroup of $\Q_+^\times$ generated by $S$ is denoted by $H$. Note that the submonoid of $\N^\times$ generated by $S$, which we refer to as $H^+$, forms a positive cone inside $H$. As the elements in $S$ are  relatively prime, $H^+$ is isomorphic to the free abelian monoid in $\lvert S \rvert$ generators. Finally, we let $H_k$ be the subgroup of $H$ generated by the $k$ smallest elements of $S$ for $1 \leq k \leq \lvert S \rvert$, and define $H_k^+$ as the analogous submonoid of $H^+$.

Though the natural action of $H^+$ on $\Z$ given by multiplication is irreversible, it has a natural extension to an action of $H$ by automorphisms, namely by acting upon the ring extension $\Z\bigl[\{ \frac{1}{p} : p \in S\}\bigr] = \Z\bigl[\{ \frac{1}{p} : p \in \PS\}\bigr]$, that will be denoted by $N$. Within this context we will consider the collection of cosets
\[\begin{array}{c} \F := \{m+ h\Z : m \in \Z, h \in \N^\times, \frac{1}{h} \in N\}. \end{array}\]

\begin{defn}\label{def:Q_S}
$\QQ_S$ is defined to be the universal $C^*$-algebra generated by a unitary $u$ and isometries $(s_p)_{p \in S}$ subject to the relations:
\[\begin{array}{llllll} \textnormal{(i)} & s_p^*s_q^{\phantom{*}} = s_q^{\phantom{*}}s_p^*, \quad & \textnormal{(ii)} & s_pu = u^ps_p, \quad \text{and} \quad & \textnormal{(iii)} & \sum\limits_{m=0}^{p-1} e_{m+p\Z} = 1 
\end{array}\]
for all $p,q \in S, p \neq q$, where $e_{m+p\Z} = u^ms_p^{\phantom{*}}s_p^*u^{-m}$.
\end{defn}

Observe that the notation $e_{m+p\Z}$ is unambiguous, i.e.\ it does not depend on the representative $m$ of the coset $m+p\Z$, as
\[
u^{m+pn}s_p^{\phantom{*}}s_p^*u^{-m-pn} \stackrel{(ii)}{=} u^ms_p^{\phantom{*}}u^{n-n}s_p^*u^{-m} = u^ms_p^{\phantom{*}}s_p^*u^{-m}.
\]

\begin{rem}\label{rem:Q_S basic I}
Let us briefly discuss the defining relations for $\QQ_S$:
\begin{enumerate}[a)]
\item Condition~(i) is known as the double commutation relation for the isometries $s_p$ and $s_q$ with $p \neq q$. In particular, they commute as $s_p^*s_q^{\phantom{*}} = s_q^{\phantom{*}}s_p^*$ implies that $(s_ps_q)^*s_q^{\phantom{*}}s_p^{\phantom{*}} = 1$, which forces $s_ps_q=s_qs_p$. Thus the family $(s_p)_{p \in S}$ gives rise to a representation of the monoid $H^+$ by isometries and we write $s_h$ for $s_{p_1}\cdots s_{p_n}$ whenever $h = p_1 \cdots p_n \in H^+$ with $p_i \in S$. In fact, $u$ and $(s_p)_{p \in S}$ yield a representation of $\Z \rtimes H^+$ due to Definition~\ref{def:Q_S}~(i),(ii).

\item $\QQ_S$ can also be defined as the universal $C^*$-algebra generated by a unitary $u$ and isometries $(s_p)_{p \in H^+}$ subject to (ii), (iii) and
\[
 \textnormal{(i') } s_ps_q = s_{pq} \text{ for all } p,q \in H^+.
\]
By a), we only need to show that (i') implies (i) for $p \neq q$. Note that (i') and (ii) imply that (iii) holds for all $p \in H^+$. In addition, (iii) implies the following: If $r \in S$ and $k \in \Z$ satisfy $s_r^*u^ks_r^{\phantom{*}} \neq 0$, then $k \in r\Z$. As $pq=qp$ and $p\Z \cap q\Z = pq\Z$, we get
\[\begin{array}{lclclclcl}
s_p^*s_q^{\phantom{*}} 
&\stackrel{(iii)}{=}& \sum\limits_{k=0}^{pq-1} s_p^*u^ks_{pq}^{\phantom{*}}s_{pq}^*u^{-k}s_q^{\phantom{*}}
&\stackrel{(i')}{=}& \sum\limits_{k=0}^{pq-1} (s_p^*u^ks_p^{\phantom{*}})s_q^{\phantom{*}}s_p^*(s_q^*u^{-k}s_q^{\phantom{*}})
&=& s_q^{\phantom{*}}s_p^*.
\end{array}\]
\end{enumerate}
\end{rem}

\begin{rem}\label{rem:can rep}
The $C^*$-algebra $\QQ_S$ has a canonical representation on $\ell^2(\Z)$: Let $(\xi_n)_{n \in \Z}$ denote the standard orthonormal basis for $\ell^2(\Z)$. If we define $U\xi_n := \xi_{n+1}$ and $S_p\xi_n := \xi_{pn}$, then it is routine to verify that $U$ and $(S_p)_{p \in S}$ satisfy (i)--(iii) from Definition~\ref{def:Q_S}. $\QQ_S$ is known to be simple, see \cite{Sta1}*{Example~3.29~(a) and Proposition~3.2} for proofs and Proposition~\ref{prop:Q_S as O-alg} for the connection to \cite{Sta1}. Therefore, the representation from above is faithful and $\QQ_S$ can be regarded as a subalgebra of $B(\ell^2(\Z))$.
\end{rem}

\begin{rem}\label{rem:Q_S for 2 and N}
For the case of $S=\{\text{all primes}\}$, the algebra $\QQ_S$ coincides with $\QQ_\N$ as introduced by Cuntz in \cite{CuntzQ}. Moreover, $S=\{2\}$ yields the \emph{$2$-adic ring $C^*$-algebra of the integers} that has been studied in detail by Larsen and Li in \cite{LarsenLi}.
\end{rem}

\begin{defn} \label{def:D_S}
The commutative subalgebra of $\QQ_S$ generated by the projections $e_{m+h\Z}= u^ms_h^{\phantom{*}}s_h^*u^{-m}$ with $m \in \Z$ and $h \in H^+$ is denoted by $\DD_S$.
\end{defn}

\begin{rem}\label{rem:Q_S basic II}
We record the following observations:
\begin{enumerate}[a)]
\item In view of Remark~\ref{rem:can rep}, $e_{m+h\Z}$ can be regarded as the orthogonal projection from $\ell^2(\Z)$ onto $\ell^2(m+h\Z)$.  
\item With regards to a), the projections $e_{m+h\Z}$ correspond to certain cosets from $\F$. However, projections arising as sums of such elementary projections may lead to additional cosets: If $h \in \N^\times$ belongs to the submonoid generated by $\PS$, then there is $h' \in H^+$ so that $h'=h\ell$ for some $\ell \in \N^\times$. Therefore, we get 
\[\begin{array}{c} e_{m+h\Z} = \sum\limits_{k=0}^{\ell-1} e_{m+hk + h'\Z}\end{array}\] 
and $e_{m+h\Z} \in \DD_S$ for all such $h$. In fact, $\F$ equals the collection of all cosets for which the corresponding projection appears in $\DD_S$, that is, the projection is expressible as a finite sum of projections $e_{m_i+h_i\Z}$ with $m_i \in \Z$ and $h_i \in H^+$.
\end{enumerate}
\end{rem}

\begin{defn}\label{def:B_S}
The subalgebra of $\QQ_S$ generated by $\DD_S$ and $u$ is denoted by $\BB_S$.
\end{defn}

\begin{rem}\label{rem:BD subalg B_S} 
The $C^*$-algebra $\BB_S$ is isomorphic to the Bunce-Deddens algebra of type $d^\infty$. If $p \in H^+$ and $(e_{i, j}^{(p)})_{0 \leq i,j \leq p-1}$ denote the standard matrix units in $M_p(\C)$, then there is a unital $*$-homomorphism $M_p(\C) \otimes C^*(\Z) \to \BB_S$ mapping $e^{(p)}_{m,n} \otimes u^k$ to $e_{m+p\Z}u^{m-n+pk}$. Given another $q \in H^+$, the so constructed $*$-homomorphisms associated with $p$ and $pq$ are compatible with the embedding $\iota_{p,pq}\colon M_p(\C) \otimes C^*(\Z) \to M_{pq}(\C) \otimes C^*(\Z)$ given by $e_{i, j}^{(p)} \otimes 1 \mapsto \sum_{k=0}^{q-1}e_{i+pk, j+pk}^{(pq)} \otimes 1$ and $1 \otimes u \mapsto 1 \otimes u^q$. The inductive limit associated with $(M_p(\C) \otimes C^*(\Z),\iota_{p,pq})_{p,q \in H^+}$, where $H^+ \subset \N^\times$ is directed set in the usual way, is isomorphic to the Bunce-Deddens algebra of type $d^\infty$. Moreover, under this identification, the natural UHF subalgebra $M_{d^\infty}$ of the Bunce-Deddens algebra corresponds to the $C^*$-subalgebra of $\BB_S$ generated by all elements of the form $e_{m+p\Z}u^{m-n}$ with $p\in H^+$ and $0\leq m,n \leq p-1$.

There is a natural action $\alpha$ of $\Z \rtimes H^+$ on $\BB_S$ given by  $\alpha_{(k,p)}(x) = u^ks_pxs_p^*u^{-k}$ for $(k,p) \in \Z \rtimes H^+$. Under the above identification, $M_{d^\infty} \subset \BB_S$ is invariant under the restricted $H^+$-action, as for $p,q \in H^+$ and $0\leq m,n \leq p-1$,
\[
s^{\phantom{*}}_qe^{\phantom{*}}_{m+p\Z}u^{m-n}s_q^* = s_q^{\phantom{*}} u^ms^{\phantom{*}}_ps_p^*u^{-n} s_q^* = u^{qm}s_{pq}^{\phantom{*}}s_{pq}^*u^{-qn} =e^{\phantom{*}}_{qm+pq\Z}u^{qm-qn}. 
\]
\vspace*{0mm}
\end{rem}

Another way to present the algebra $\QQ_S$ is provided by the theory of semigroup crossed products. Recall that, for an action $\beta$ of a discrete, left cancellative semigroup $T$ on a unital $C^*$-algebra $B$ by $*$-endomorphisms, a unital, covariant representation of $(B, \beta, T)$ is given by a unital $*$-homomorphism $\pi$ from $B$ to some unital $C^*$-algebra $C$ and a semigroup homomorphism $\varphi$ from $T$ to the isometries in $C$ such that the covariance condition $\varphi(t)\pi(b)\varphi(t)^* = \pi(\beta_t(b))$ holds for all $b \in B$ and $t \in T$. The semigroup crossed product $B \rtimes^{e}_\beta T$ is then defined as the $C^*$-algebra generated by a universal unital, covariant representation $(\iota_B,\iota_T)$ of $(B, \beta, T)$. We refer to \cite{LacRae} for further details. Note that if $T$ is a group, then this crossed product agrees with the full group crossed product $B \rtimes_\beta T$. Semigroup crossed products may be pathological or extremely complicated in some cases. But we will only be concerned with crossed products of left Ore semigroups acting by injective endomorphisms so that we maintain a close connection to group crossed products, see \cite{Lac}. With respect to $\QQ_S$, we get isomorphisms
\begin{equation}\label{eq:B_S and Q_S as crossed products}
\QQ_S \cong \DD_S \rtimes^e_\alpha \Z \rtimes H^+ \cong \BB_S \rtimes^e_\alpha H^+ \quad\text{and}\quad \BB_S \cong \DD_S\rtimes_\alpha \Z,
\end{equation}
see \cite{Sta1}*{Proposition~3.18 and Theorem~A.5}. Remark~\ref{rem:BD subalg B_S} reveals that the canonical subalgebra $M_{d^\infty} \subset \BB_S$ and the isometries $(s_p)_{p \in H^+}$ give rise to a unital, covariant representation of $(M_{d^{\infty}},\alpha,H^+)$ on $\QQ_S$. We will later see that this representation is faithful so that we can view $M_{d^\infty} \rtimes^e_\alpha H^+$ as a subalgebra of $\QQ_S$, see Corollary~\ref{cor:subalgebra for torsion part}.

\subsection{Boundary quotients} \label{subsec:BQ for ADS}
The set $S \subset \N^\times\setminus\{1\}$ itself can be thought of as a data encoding a dynamical system, namely the action $\theta$ of the free abelian monoid $H^+ \subset \N^\times$ on the group $\Z$ given by multiplication. $\theta_h$ is injective for $h \in H^+$ and surjective only if $h=1$. Furthermore, as every two distinct elements $p$ and $q$ in $S$ are relatively prime, we have $\theta_p(\Z)+\theta_q(\Z)=\Z$. Hence $(\Z,H^+,\theta)$ forms an \emph{irreversible algebraic dynamical system} in the sense of \cite{Sta1}*{Definition 1.5}, compare \cite{Sta1}*{Example~1.8~(a)}. In fact, dynamics of this form were one of the key motivations for \cite{Sta1}. In order to compare the $C^*$-algebra $\CO[\Z,H^+,\theta]$ from \cite{Sta1}*{Definition 3.1} with $\QQ_S$, let us recall the definition:

\begin{defn}\label{def:O-alg for IADS}
Let $(G,P,\theta)$ be an irreversible algebraic dynamical system. Then $\CO[G,P,\theta]$ is the universal $C^*$-algebra generated by a unitary representation $u$ of the group $G$ and a representation $s$ of the semigroup $P$ by isometries subject to the relations:
\[\begin{array}{lrcl}
(\text{CNP }1) & s_{p}u_{g} &\hspace*{-2.5mm}=\hspace*{-2.5mm}& u_{\theta_{p}(g)}s_{p},\vspace*{2mm}\\
(\text{CNP }2) & s_{p}^{*}u_gs_{q} &\hspace*{-2.5mm}=\hspace*{-2.5mm}& \begin{cases} 
u_{g_1}s_{\gcd(p,q)^{-1}q}^{\phantom{*}}s_{\gcd(p,q)^{-1}p}^{*}u_{g_2}& \text{ if } g = \theta_p(g_1)\theta_q(g_2),\\ 0& \text{ else,}\end{cases}\vspace*{2mm}\\
(\text{CNP }3) & 1 &\hspace*{-2.5mm}=\hspace*{-2.5mm}& \sum\limits_{[g] \in G/\theta_{p}(G)}{e_{g,p}} \hspace*{2mm}\text{ if } [G : \theta_{p}(G)]< \infty,
\end{array}\]
where $e_{g,p} = u_{g}s_{p}s_{p}^{*}u_{g}^{*}$.
\end{defn}

Clearly, Definition~\ref{def:Q_S}~(ii) is the same as (CNP $1$). As $[\Z:h\Z] = h < \infty$ for every $h \in H^+$, (iii) corresponds to (CNP $3$) once we use Remark~\ref{rem:Q_S basic I}~a) and note that it is enough to have the summation relation for a set of generators of $H^+$. The case of distinct $p,q \in S$ and $g=0$ in (CNP $2$) yields (i). On the other hand, a slight modification of the argument in Remark~\ref{rem:Q_S basic I}~b) with $s_p^*u^ms_q$ in place of $s_p^*s_q$ establishes (CNP $2$) based on (i)--(iii).
Thus we arrive at:

\begin{prop}\label{prop:Q_S as O-alg}
The $C^*$-algebras $\QQ_S$ and $\CO[\Z,H^+,\theta]$ are canonically isomorphic.
\end{prop}
According to \cite{Sta1}*{Corollary~3.28 and Example~3.29~(a)} $\QQ_S$ is therefore a unital UCT Kirchberg algebra. While classification of $\CO[G,P,\theta]$ by $K$-theory was achieved in \cite{Sta1} for irreversible algebraic dynamical systems $(G,P,\theta)$ under mild assumptions, and even generalized to \emph{algebraic dynamical systems} in \cite{bsBQforADS}, the range of the classifying invariant remained a mystery beyond the case of a single group endomorphism, where the techniques of \cite{CuntzVershik} apply. It thus seemed natural to go back to examples of dynamical systems involving $P = \N^k$ and try to understand the invariant in this case. In other words, our path lead back to $\QQ_S$, and the present work aims at making progress precisely in this direction. 
\vspace*{1em}

\noindent There is also an alternative way of constructing $\QQ_S$ directly from either of the semigroups $\N \rtimes H^+$ or $\Z \rtimes H^+$ using the theory of boundary quotients of semigroup $C^*$-algebras. To begin with, let us note that $(\N \rtimes H^+,N \rtimes H)$ forms a quasi lattice-ordered group. Hence we can form the Toeplitz algebra $\mathcal{T}(\N \rtimes H^+,N \rtimes H)$ using the work of Nica, see \cite{Nic}. But $\Z \rtimes H^+$ has non-trivial units, so it cannot be part of a quasi lattice-ordered pair. In order to treat both semigroups within the same framework, let us instead employ the theory of semigroup $C^*$-algebras from \cite{Li1}, which generalizes Nica's approach tremendously. 

We note that both $\N \rtimes H^+$ and $\Z \rtimes H^+$ are cancellative, countable, discrete semigroups with unit. Moreover, they are \emph{right LCM} semigroups, meaning that the intersection of two principal right ideals is either empty or another principal right ideal (given by a right least common multiple for the representatives of the two intersected ideals). Thus their semigroup $C^*$-algebras both enjoy a particularly nice and tractable structure, see \cites{BLS1,BLS2}. Additionally, both are left Ore semigroups with amenable enveloping group $N \rtimes H \subset \Q \rtimes \Q_+^\times$. However, we would like to point out that $\N \rtimes H^+$ and $\Z \rtimes H^+$ are not left amenable (but right amenable) as they fail to be left reversible, see \cite{Li1}*{Lemma~4.6} for details.

Roughly speaking, semigroup $C^*$-algebras have the flavor of Toeplitz algebras. In particular, they tend to be non-simple except for very special situations. Still, we might hope for $\QQ_S$ to be a quotient of $C^*(\N \rtimes H^+)$ or $C^*(\Z \rtimes H^+)$ obtained through some systematic procedure. This was achieved in \cite{LacRae} for $\N \rtimes \N^\times$, i.e.\ $S$ consisting of all primes, by showing that the boundary quotient of $\mathcal{T}(\N \rtimes \N^\times,\Q \rtimes \Q_+^\times) = C^*(\N \rtimes \N^\times)$ in the sense of \cite{CrispLaca} coincides with $\QQ_\N$. Recently, this concept of a boundary quotient for a quasi lattice-ordered group from \cite{CrispLaca} was transferred to semigroup $C^*$-algebras in the context of right LCM semigroups, see \cite{BRRW}*{Definition 5.1}: 

\begin{defn}\label{def:FS and boundary quotient}
Let $T$ be a right LCM semigroup. A finite set $F \subset T$ is called a \emph{foundation set} if, for all $t\in T$, there is $f \in F$ satisfying $tT \cap fT \neq \emptyset$. The \emph{boundary quotient} $\QQ(T)$ of a right LCM semigroup $T$ is the quotient of $C^*(T)$ by the relation 
\begin{equation}\label{eq:BQ}
\begin{array}{c}\prod\limits_{f \in F} (1-e_{fT}) = 0 \quad \text{ for all foundation sets } F.\end{array}
\end{equation}
\end{defn}

To emphasize the relevance of this approach, let us point out that right LCM semigroups are much more general than quasi lattice-ordered groups. For instance, right cancellation may fail, so right LCM semigroups need not embed into groups.

On the one hand, this notion of a quotient of a semigroup $C^*$-algebra seems suitable as $\N \rtimes H^+$ and $\Z \rtimes H^+$ are right LCM semigroups. On the other hand, the abstract condition~\eqref{eq:BQ} prohibits an immediate identification of $\QQ_S$ with $\QQ(\N \rtimes H^+)$ or $\QQ(\Z \rtimes H^+)$. This gap has been bridged successfully through \cite{bsBQforADS}: 

\begin{prop}\label{prop:Q_S as BQ}
There are canonical isomorphisms $\QQ_S \cong \QQ(\Z \rtimes H^+) \cong \QQ(\N \rtimes H^+)$.
\end{prop}
\begin{proof}
For $\Z \rtimes H^+$, \cite{bsBQforADS}*{Corollary~4.2} shows that $\QQ(\Z \rtimes H^+) \cong \CO[\Z,H^+,\theta]$, and hence $\QQ_S \cong \QQ(\Z \rtimes H^+)$ by Proposition~\ref{prop:Q_S as O-alg}. Noting that $H^+$ is directed, this can also be seen immediately from \cite{bsBQforADS}*{Remark~2.2 and Proposition~4.1}. For $\N\rtimes H^+$, we infer from \cite{bsBQforADS}*{Example~2.8~(b)} that it suffices to consider \emph{accurate foundation sets} $F$ for \eqref{eq:BQ} by \cite{bsBQforADS}*{Proposition~2.4}, that is, $F$ consists of elements with mutually disjoint principal right ideals. Now $F \subset \N \rtimes H^+$ is an accurate foundation set if and only if it is an accurate foundation set for $\Z \rtimes H^+$. Conversely, given an accurate foundation set $F' =\{(m_1,h_1),\dots,(m_n,h_n)\} \subset \Z \rtimes H^+$, we can replace each $m_i$ by some $m_i' \in m_i+h_i\N$ with $m_i' \in \N$ to get an accurate foundation set $F \subset \N \rtimes H^+$ which uses the same right ideals as $F'$. This allows us to conclude that $\QQ(\Z \rtimes H^+)$ and $\QQ(\N \rtimes H^+)$ are isomorphic.
\end{proof}
 
The fact that $\QQ(\N \rtimes H^+)$ and $\QQ(\Z \rtimes H^+)$ coincide is not at all surprising if we take into account \cite{BaHLR} and view $C^*(\Z \rtimes H^+)$ as the \emph{additive boundary quotient} of $\N \rtimes H^+$. Where there is an additive boundary, there is also a multiplicative boundary, see the boundary quotient diagram in \cite{BaHLR}*{Section~4}: The \emph{multiplicative boundary quotient} of $C^*(\N \rtimes H^+)$ is obtained by imposing the analogous relation to (iii) from Definition~\ref{def:Q_S}, i.e.\ $\sum_{k=0}^{p-1} e_{k+p\N} = 1$ for each $p \in S$. In comparison with Definition~\ref{def:Q_S}, the essential difference is that the semigroup element $(1,1)$ is implemented by a proper isometry $v_{(1,1)}$ instead of a unitary $u$. This multiplicative boundary quotient has been considered in \cite{LN2}*{Subsection~6.5}. As it turns out, its $K$-theory is hard to compute for larger $S$  as it leads to increasingly complicated extension problems of abelian groups. It is quite remarkable that there seems to be a deep common theme underlying the structure of the $K$-theory for both the multiplicative boundary quotient and $\QQ_S$.

\subsection{\texorpdfstring{The $a$-adic algebras}{The a-adic algebras}} \label{subsec:a-adic algs}

Our aim is to compute the $K$-theory of $\QQ_S$, and for this we need to make use of a certain duality result \cite{KOQ}*{Theorem~4.1} that allows us to translate our problem into real dynamics. This will be explained in the next section, but let us first recall the definition and some facts about $a$-adic algebras from \cite{KOQ} and \cite{Oml}, see also \cite{hr}*{Sections~10 and~25} for more on $a$-adic numbers.

Let $a=(a_k)_{k\in\Z}$ be a sequence of numbers in $\N^\times\setminus\{1\}$, and define the \emph{$a$-adic numbers} as the abelian group of sequences
\[\begin{array}{c}
\Omega_a = \left\{ x\in\prod\limits_{k=-\infty}^{\infty}\{0,1,\dotsc,a_k-1\} : \text{there exists $\ell\in\Z$ such that $x_k=0$ for all $k<\ell$}\right\}
\end{array}\]
under addition with carry (that is, like a doubly infinite odometer). The family of all subgroups $\{x\in\Omega_a:x_k=0\text{ for }k<\ell\}$ form a neighborhood basis of the identity. This induces a topology that makes $\Omega_a$ a totally disconnected, locally compact Hausdorff group. The \emph{$a$-adic integers} is the compact open subgroup
\begin{equation}\label{eq:D-spectrum}
\Delta_a=\{x\in\Omega_a:x_k=0\text{ for }k<0\}\subset\Omega_a.
\end{equation}
For $k \in \Z$, define the sequence $(e_k)_\ell=\delta_{k \ell}$. For $k\geq 1$, we may associate the rational number $(a_{-1}a_{-2}\dotsm a_{-k})^{-1}$ with $e_{-k}$ to get an injective group homomorphism from the non-cyclic subgroup
\[\begin{array}{c}
N_a=\left\{\frac{j}{a_{-1}a_{-2}\cdots a_{-k}}:j\in\Z, k\geq 1\right\} \subset \Q
\end{array}\]
into $\Omega_a$ with dense range. Note that $N_a$ contains $\Z \subset \Q$, and by identifying $N_a$ and $\Z \subset N_a$ with their images under the embedding into $\Omega_a$, it follows that $N_a \cap\Delta_a = \Z$.

The subgroups $N_a \cap\{x\in\Omega_a:x_k=0\text{ for }k<\ell\}$ for $\ell\in\Z$ give rise to a subgroup topology of $N_a$, and $\Omega_a$ is the Hausdorff completion (i.e.\ inverse limit completion) of $N_a$ with respect to this filtration. Therefore, the class of $a$-adic numbers $\Omega_a$ comprises all groups that are Hausdorff completions of non-cyclic subgroups of $\Q$. Loosely speaking, the negative part of the sequence $a$ determines a subgroup $N_a$ of $\Q$, and the positive part determines a topology that gives rise to a completion of $N_a$. Given a sequence $a$, let $a^*$ denote the dual sequence defined by $a^*_k=a_{-k}$, and write $N_a^*$ and $\Omega_a^*$ for the associated groups.

Let $H_a$ be any non-trivial subgroup of $\Q^\times_+$ acting on $N_a$ by continuous multiplication, meaning that for all $h\in H_a$, the map $N_a\to N_a$, $x\mapsto hx$ is continuous with respect to the topology described above. The largest subgroup with this property is generated by the primes dividing infinitely many terms of both the positive and negative tail of the sequence $a$, see \cite{KOQ}*{Corollary~2.2}, so we must assume that this subgroup is non-trivial (which holds in the cases we study). Then $H_a$ also acts on $\Omega_a$ by multiplication, and therefore $N_a\rtimes H_a$ acts on $\Omega_a$ by an $ax+b$-action.

\begin{defn}\label{def:Q_a,H}
For a sequence $a=(a_k)_{k \in \Z}$ in $\N^\times\setminus\{1\}$ and a non-trivial subgroup $H_a$ of $\Q^\times_+$ acting by continuous multiplication on $N_a$, the crossed product $\overline{\QQ}(a,H_a) := C_0(\Omega_a)\rtimes N_a\rtimes H_a$ is called the \emph{$a$-adic algebra} of $(a,H_a)$.
\end{defn}

Clearly, interchanging $a$ and $a^*$ and manipulating the position of $a_0$ will not affect any structural property on the level of algebras. In fact, for our purposes, it will usually be convenient to assume that $a=a^*$. Therefore, we will often use the positive tail of the sequence $a$ in the description of $N_a$, and think of $N_a$ as the inductive limit of the system $\left\lbrace (\Z ,\cdot a_k) : k \geq 0 \right\rbrace$ via the isomorphism induced by
\begin{equation}\label{eq:N_a ind lim}
\begin{gathered}
\begin{xy}
	\xymatrix{
	\Z \ar[rr]^{\cdot a_k} \ar[rd]_(0.3){\cdot \tfrac{1}{a_0a_1a_2\dotsm a_{k-1}}} & & \Z \ar[ld]^(0.3){\cdot \tfrac{1}{a_0a_1a_2\dotsm a_{k-1}a_k}}\\
		& N_a &	
	}
\end{xy}
\end{gathered}
\end{equation}

\begin{rem}\label{rem:a-adic stability}
By \cite{KOQ}*{Corollary~2.8} the $a$-adic algebra $\overline{\QQ}(a,H)$ is always a non-unital UCT Kirchberg algebra, hence it is stable by Zhang's dichotomy, see \cite{Z} or \cite{rordam-zd}*{Proposition~4.1.3}.

An immediate consequence of \eqref{eq:N_a ind lim} is that
\[
\begin{array}{c}
C_0(\Omega_a)\rtimes N_a \cong \overline{\bigcup\limits_{k=0}^\infty C(\frac{1}{a_0\dotsm a_k}\Delta_a)\rtimes \frac{1}{a_0\dotsm a_k}\Z}.
\end{array}
\]
Moreover, by writing $\frac{1}{a_0\dotsm a_k}\Delta_a=\Delta_a+(\frac{a_0\dotsm a_k - 1}{a_0\dotsm a_k}+\Delta_a)+\dotsb+(\frac{1}{a_0\dotsm a_k}+\Delta_a)$ and checking how the translation action of $\frac{1}{a_0\dotsm a_k}\Z$ interchanges the components of this sum, one sees that
\[
\begin{array}{c}
C(\frac{1}{a_0\dotsm a_k}\Delta_a)\rtimes \frac{1}{a_0\dotsm a_k}\Z \cong M_{a_0\dotsm a_k}\left(C(\Delta_a)\rtimes\Z\right).
\end{array}
\]
In particular, the natural embeddings of the increasing union above translates into embeddings into the upper left corners. Hence, it follows that $C_0(\Omega_a)\rtimes N_a$ is also stable.
\end{rem}

\begin{rem}\label{rem:supernatural}
A supernatural number is a function $d\colon\{\text{all primes}\} \to \N \cup \{\infty\}$, such that $\sum_{\text{$p$ prime}}d(p) = \infty$, and often written as a formal product $\prod_{\text{$p$ prime}} p^{d(p)}$. It is well known that there is a one-to-one correspondence between supernatural numbers and non-cyclic subgroups of $\Q$ containing $1$, and that the supernatural numbers form a complete isomorphism invariant both for the UHF algebras and the Bunce-Deddens algebras, see \cite{Gl} and \cite{BD}.

Every sequence $a=(a_k)_{k\geq 0}$ defines a function $d_a\colon\{\text{all primes}\} \to \N \cup \{\infty\}$ given by $d_a(p)=\sup\{n \in \N : p^n|a_0a_1 \dotsm a_k \text{ for some } k\geq 0\}$. More intuitively, $d_a$ is thought of as the infinite product $d_a=a_0a_1a_2\dotsm$. Moreover (see e.g.\ \cite{KOQ}*{Lemma~5.1}), we have
\begin{equation}\label{eq:a-adic int}
\begin{array}{c}
\Delta_a\cong\prod\limits_{p\in d_a^{-1}(\infty)}\Z_p\times\prod\limits_{p\in d_a^{-1}(\N)}\Z/p^{d_a(p)}\Z,
\end{array}
\end{equation}
and thus the supernatural numbers are a complete isomorphism invariant for the homeomorphism classes of $a$-adic integers.
\end{rem}

Now, as in Section~\ref{sec:prelim}, let $S$ be a set consisting of relatively prime numbers, and let $H^+$ and $H$ denote the submonoid of $\N^\times$ and the subgroup of $\Q^\times_+$ generated by $S$, respectively. The sequence $a_S$ is defined as follows: Since $H^+$ is a subset of $\N^\times$, its elements can be sorted into increasing order $1<a_{S,0}<a_{S,1}<\dotsb$, where $a_{S,0}=\min S$. Finally, we set $a_{S,k}=a_{S,-k}$ for $k<0$. If $S$ is a finite set, an easier way to form a suitable sequence $a_S$ is to let $q$ denote the product of all elements of $S$, and set $a_{S,k}=q$ for all $k\in\Z$. In both cases, $a_S^*=a_S$ and $N_{a_S}=N$. Henceforth we fix such a sequence $a_S$ and denote $\Omega_{a_S}$ and $\Delta_{a_S}$ by $\Omega$ and $\Delta$, respectively. The purpose of self-duality of $a_S$ is to have $N^*=N$, making the statement of Theorem~\ref{thm:duality} slightly more convenient by avoiding the explicit use of $N^*$. The sequences $a_S$ are the ones associated with supernatural numbers $d$ for which $d(p) \in \{0,\infty\}$ for every prime $p$. In this case \eqref{eq:a-adic int} implies that
\[\begin{array}{c}
\Delta \cong\prod\limits_{p \in \PS}\Z_p \quad\text{and}\quad
\Omega \cong\prod\limits_{p \in \PS}\hspace*{-2mm}' \hspace*{2mm} \Q_p=\prod\limits_{p \in \PS}(\Q_p,\Z_p),  
\end{array}\]
where the latter denotes the restricted product with respect to $\left\lbrace \Z_p : p\in \PS \right\rbrace$.

\begin{rem}\label{rem:spec of D_S}
The spectrum of the commutative subalgebra $\DD_S$ of $\QQ_S$ from Definition~\ref{def:D_S} coincides with the $a$-adic integers $\Delta$ described in \eqref{eq:D-spectrum}. Indeed, for every $X=m+h\Z\in\F$, the projection $e_X$ in $\DD_S$ corresponds to the characteristic function on the compact open subset $m+h\Delta$ of $\Delta$. Moreover, this correspondence extends to an isomorphism between the $C^*$-algebra $\BB_S\cong\DD_S\rtimes\Z$ of Definition~\ref{def:B_S} and $C(\Delta)\rtimes\Z$, which is equivariant for the natural $H^+$-actions on the algebras, both denoted by $\alpha$.
\end{rem}

Let us write $e$ for the projection in $\overline{\QQ}(a_S,H)$ representing the characteristic function on $\Delta$ in $C_0(\Omega)$. It is explained in \cite{Oml}*{Section~11.6} that $e$ is a full projection, and thus, by using Remark~\ref{rem:spec of D_S} together with \eqref{eq:B_S and Q_S as crossed products}, we have
\begin{equation}\label{eq:full corner Q_S}
e\overline{\QQ}(a_S,H)e
\cong (C(\Delta) \rtimes \Z) \rtimes^e_\alpha H^+
\cong (\DD_S \rtimes \Z) \rtimes^e_\alpha H^+
\cong \BB_S \rtimes^e_\alpha H^+
\cong \QQ_S.
\end{equation}
In fact, since $N$ coincides with $(H^+)^{-1}\Z$, the above also follows from \cite{Lac}. Moreover, the argument in \cite{Oml} does not require $H$ to be non-trivial, so it can be used together with Remark~\ref{rem:spec of D_S} and \eqref{eq:B_S and Q_S as crossed products} to get
\begin{equation}\label{eq:full corner B_S}
e(C_0(\Omega) \rtimes N)e
\cong C(\Delta) \rtimes \Z
\cong \DD_S \rtimes \Z
\cong \BB_S.
\end{equation}
Hence, by applying Remark~\ref{rem:a-adic stability} we arrive at the following result:
\begin{prop}\label{prop:stable Q_S}
The stabilization of $\QQ_S$ is isomorphic to $\overline{\QQ}(a_S,H)$, and the stabilization of $\BB_S$ is isomorphic to $C_0(\Omega)\rtimes N$.
\end{prop}
Therefore Proposition~\ref{prop:stable Q_S} gives an alternative way to see that $\QQ_S$ is a unital UCT Kirchberg algebra, which is also a consequence of Proposition~\ref{prop:Q_S as O-alg}.

\section{Comparison with real dynamics}\label{sec:comp with real dyn}

Let $S$ and $H_k$ be as specified in Section~\ref{sec:prelim}, $a=(a_k)_{k \in \Z}$ in $\N^\times \setminus\{1\}$, and $H_a$ a non-trivial subgroup of $\Q_+^\times$ that acts on $N_a$ by continuous multiplication. For convenience, we will assume $a^*= a$ so that $N_a^*=N_a$. Moreover, $N_a$ acts by translation and $H_a$ acts by multiplication on $\R$, respectively, giving rise to an $ax+b$-action of $N_a\rtimes H_a$ on $\R$. Let $\widehat{N}_a$ denote the Pontryagin dual of $N_a$. By \cite{KOQ}*{Theorem~3.3}, the diagonal embedding $N_a \to\R\times\Omega_a$ has discrete range, and gives an isomorphism
\[(\R\times\Omega_a)/N_a \cong \widehat{N}_a.\]
By applying Green's symmetric imprimitivity theorem, see e.g.\ \cite{Wil}*{Corollary~4.11}, we obtain that
\[C_0(\Omega_a) \rtimes N_a \sim_M C_0(\R) \rtimes N_a,\]
and this Morita equivalence is equivariant for the actions of $H_a$ by multiplication on one side and inverse multiplication on the other. The inverse map on $H_a$ does not have any impact on the crossed products, and thus
\[\overline{\QQ}(a,H_a) \sim_M C_0(\R) \rtimes N_a \rtimes H_a.\]
All the above is explained in detail in \cite{KOQ}*{Proof of Theorem~4.1}. Moreover, recall that UCT Kirchberg algebras are either unital or stable, so by using Proposition~\ref{prop:stable Q_S} we get:
\begin{thm}\label{thm:duality}
The $a$-adic algebra $\overline{\QQ}(a,H_a)$ is isomorphic to $C_0(\R) \rtimes N_a \rtimes H_a$. In particular, the stabilization of $\QQ_S$ is isomorphic to $C_0(\R) \rtimes N \rtimes H$.
\end{thm}

\begin{rem}\label{rem:generality}
It follows from Theorem~\ref{thm:duality}, based on \cite{KOQ}*{Theorem~4.1}, that any $a$-adic algebra $\overline{\QQ}(a,H_a)$ is isomorphic to a crossed product $C_0(\R)\rtimes N_a \rtimes H_a$. Recall that $N_a$ can be any non-cyclic subgroup of $\Q$ and $H_a$ can be any non-trivial subgroup of $\Q_+^\times$ that acts on $N_a$ by multiplication. In the present work, we limit our scope to the case where $N_a$ and $H_a$ can be obtained from a family $S$ of relatively prime numbers for the benefit of a more concise exposition. In a forthcoming project, we aim at establishing analogous results to the ones proven here for all $a$-adic algebras.
\end{rem}

\begin{rem}\label{rem:ind limit C_0(R) cross N}
By employing the description of $N_a$ from \eqref{eq:N_a ind lim}, we can write $C_0(\R) \rtimes N_a$ as an inductive limit. For $k \geq 0$, define the automorphism $\gamma_k$ of $C_{0}(\R)$ by
\[\begin{array}{c} 
\gamma_0(f)(s)=f(s-1) \text{ and }
\gamma_{k+1}(f)(s)=f\bigl(s-\frac{1}{a_0a_1\dotsm a_k}\bigr),\quad f\in C_0(\R). \end{array}\]
Under the identification in \eqref{eq:N_a ind lim}, these automorphisms give rise to the natural $N_a$-action on $C_0(\R)$, where $\gamma_k$ corresponds to the generator for the $k$th copy of $\Z$. For $k \geq 0$, let $u_k \in \M(C_0(\R) \rtimes_{\gamma_k} \Z)$ denote the canonical unitary implementing $\gamma_k$ and consider the $*$-homomorphism $\phi_k\colon C_0(\R) \rtimes_{\gamma_k} \Z \to C_0(\R) \rtimes_{\gamma_{k+1}} \Z$ given by $\phi_k(f) = f$ and $\phi_k(fu_k) = fu_{k+1}^{a_k}$ for every $f\in C_0(\R)$. The inductive limit description \eqref{eq:N_a ind lim} of $N_a$ now yields an isomorphism $\varphi\colon\varinjlim \left\lbrace C_0(\R) \rtimes_{\gamma_k} \Z,\phi_k \right\rbrace \stackrel{\cong}{\longrightarrow} C_0(\R) \rtimes N_a$.
\end{rem}
\begin{rem}\label{rem:duality B_S}
A modification of \cite{CuntzQ}*{Lemma~6.7}, using the inductive limit description from Remark~\ref{rem:ind limit C_0(R) cross N}, shows that $C_0(\R) \rtimes N_a$ is stable. Hence, it follows from the above together with Remark~\ref{rem:a-adic stability} that $C_0(\Omega_a) \rtimes N_a$ is isomorphic to $C_0(\R) \rtimes N_a$. In particular, Proposition~\ref{prop:stable Q_S} shows that the stabilization of $\BB_S$ is isomorphic to $C_0(\R) \rtimes N$.
\end{rem}
We will make use of this fact below.
\begin{lem} \label{lem:torsion algebra Morita}
Let $\widetilde{\alpha}$ and $\beta$ denote the actions of $H_a$ on $C_0(\Omega_a)\rtimes N_a$ and $C_0(\R)\rtimes N_a$, respectively. Then $\beta^{-1}$ is exterior equivalent to an action $\widetilde{\beta}$ for which there is an $\widetilde{\alpha}$ - $\widetilde{\beta}$-equivariant isomorphism $C_0(\Omega_a) \rtimes N_a \stackrel{\cong}{\longrightarrow} C_0(\R) \rtimes N_a$.
\end{lem}
\begin{proof}
The respective actions $\widetilde{\alpha}$ and $\beta^{-1}$ of $H_a$ are Morita equivalent by \cite{KOQ}*{Proof of Theorem~4.1}. Moreover, both $C^*$-algebras are separable and stable, see Remark~\ref{rem:duality B_S}. Therefore, \cite{Com}*{Proposition on p.~16} implies that the actions are also outer conjugate, and the statement follows.
\end{proof}

In the following, we denote by $\iota_{N_a}\colon C_0(\R) \into C_0(\R) \rtimes N_a$ the canonical embedding, which is equivariant for the respective $H_a$-actions $\beta$ (and also $\beta^{-1}$). We conclude this section by proving that $\iota_{N_a}$ induces an isomorphism between the corresponding $K_1$-groups.

\begin{prop} \label{prop:KtheoryA0}
The canonical embedding $\iota_{N_a}\colon C_0(\R) \into  C_0(\R)\rtimes N_a$ induces an isomorphism between the corresponding $K_1$-groups.
\end{prop}
\begin{proof}
Recall the isomorphism $\varphi\colon\varinjlim \left\lbrace C_0(\R) \rtimes_{\gamma_k} \Z,\phi_k \right\rbrace \stackrel{\cong}{\longrightarrow} C_0(\R) \rtimes N_a$ from Remark~\ref{rem:ind limit C_0(R) cross N}. For $k \geq 0$, let $\iota_k\colon C_0(\R)\to C_0(\R)\rtimes_{\gamma_k}\Z$ be the canonical embedding. As $\iota_{k+1} = \phi_k \circ \iota_k$, we obtain the following commutative diagram
\begin{equation}\label{eq:ind lim C_0(R)CrossN_a}
\begin{gathered}
\begin{xy}
\xymatrix{
C_0(\R) \ar[r]^{\iota_{N_a}} \ar[d]_{\phi_{k,\infty} \circ \ \iota_k} & C_0(\R) \rtimes N_a \\
 \varinjlim \left\lbrace C_0(\R) \rtimes_{\gamma_m} \Z,\phi_m \right\rbrace \ar[ur]_\varphi
}
\end{xy}
\end{gathered}
\end{equation}
Here, $\phi_{k,\infty}\colon C_0(\R) \rtimes_{\gamma_k} \Z \to \varinjlim \left\lbrace C_0(\R) \rtimes_{\gamma_m} \Z,\phi_m \right\rbrace$ denotes the canonical $*$-homomorphism given by the universal property of the inductive limit.

As $K_0(C_0(\R)) = 0$, the Pimsner-Voiculescu sequence \cite{PV} for $\gamma_k \in \Aut(C_0(\R))$ reduces to an exact sequence
\[
K_0(C_0(\R) \rtimes_{\gamma_k} \Z) \into K_1(C_0(\R)) \stackrel{\id - K_1(\gamma_k)}{\longrightarrow} K_1(C_0(\R)) \stackrel{K_1(\iota_k)}{\onto} K_1(C_0(\R) \rtimes_{\gamma_k} \Z).
\]
For each $k\geq 0$, the automorphism $\gamma_k$ is homotopic to the identity on $\R$, so that $K_1(\gamma_k) = \id$. It thus follows that $K_1(\iota_k)$ is an isomorphism. As $\iota_{k+1} = \phi_k \circ \iota_k$, we therefore get that $K_1(\phi_k)$ is an isomorphism as well. Hence, by continuity of $K$-theory, $K_1(\phi_{k,\infty})$ is an isomorphism. It now follows from \eqref{eq:ind lim C_0(R)CrossN_a} that $K_1(\iota_{N_a})$ is an isomorphism, which completes the proof.
\end{proof}

\section{\texorpdfstring{A decomposition of the $K$-theory of $\QQ_S$}{A decomposition of the K-theory}} \label{sec:K-theory}
In this section, we show that $K_*(\QQ_S)$ decomposes as a direct sum of a free abelian group and a torsion group, see Theorem~\ref{thm:decomposition of K-theory} and Corollary~\ref{cor:torsion and free part K-theory}. We would like to highlight that this is not just an abstract decomposition of $K_*(\QQ_S)$, but a result that facilitates a description of the two parts by distinguished $C^*$-algebras associated to $S$, namely $M_{d^\infty} \rtimes^e_\alpha H^+$ for the torsion part, and $C_0(\R) \rtimes_\beta H$ for the free part. The free abelian part is then shown to have rank $2^{\lvert S\rvert - 1}$, see Proposition \ref{prop:K-theory torsion free part}, so that $\QQ_S$ and $\QQ_T$ can only be isomorphic if $S$ and $T$ have the same cardinality.

The following is the key tool for the proof of this section's main result, and we think it is of interest in its own right.

\begin{prop} \label{prop:splitted K-theory}
Let $k \in \N\cup \left\lbrace \infty \right\rbrace$, $A,B,C$ $C^*$-algebras, and $\alpha\colon \Z^k \curvearrowright A$, $\beta\colon \Z^k \curvearrowright B$, and $\gamma\colon \Z^k \curvearrowright C$ actions. Let $v\colon \Z^k \to \U(\M(C))$ be a $\gamma$-cocycle and denote by $\tilde{\gamma}\colon \Z^k \curvearrowright C$ the induced action given by $\tilde{\gamma}_h = \op{Ad}(v_h) \circ \gamma_h$ for $h \in \Z^k$. Let $\kappa\colon C \rtimes_{\tilde{\gamma}} \Z^k \stackrel{\cong}{\longrightarrow} C \rtimes_{\gamma} \Z^k$ be the $*$-isomorphism induced by the $\gamma$-cocycle $v$. Assume that $\varphi\colon A \to C$ is a non-degenerate $\alpha$ - $\gamma$-equivariant $*$-homomorphism and $\psi\colon B\to C$ a non-degenerate $\beta$ - $\tilde{\gamma}$-equivariant $*$-homomorphism such that $K_0(\varphi)$ and $K_1(\psi)$ are isomorphisms and $K_1(\varphi)$ and $K_0(\psi)$ are trivial. Then
\[
K_*(\varphi \rtimes \Z^k) \oplus K_*(\kappa \circ (\psi \rtimes \Z^k))\colon  K_*(A \rtimes_\alpha \Z^k) \oplus K_*(B \rtimes_\beta \Z^k) \to K_*(C \rtimes_\gamma \Z^k)
\]
is an isomorphism.
\end{prop}
\begin{proof}
Consider the amplified action $\gamma^{(2)}\colon  \Z^k \curvearrowright M_2(C)$ given by entrywise application of $\gamma$. Let $w\colon \Z^k \to \U(\M(M_2(C)))$ be the $\gamma^{(2)}$-cocycle given by $w_h= \op{diag}(1,v_h)$ for $h \in \Z^k$. The induced $\Z^k$-action $\delta = \op{Ad}(w) \circ \gamma^{(2)}$ satisfies $\delta_h(\op{diag}(c,c')) = \op{diag}(\gamma_h(c),\tilde{\gamma}_h(c'))$ for all $h \in \Z^k$ and $c,c' \in C$. Thus, $\eta = \varphi \oplus \psi\colon  A \oplus B \to M_2(C)$ is a non-degenerate $\alpha \oplus \beta$ - $\delta$-equivariant $*$-homomorphism.

By additivity of $K$-theory, $K_*(\eta) = K_*(\varphi) + K_*(\psi)$. Hence, $K_*(\eta)$ is an isomorphism, as $K_0(\eta) = K_0(\varphi)$ and $K_1(\eta) = K_1(\psi)$. If $k \in \N$, an iterative use of the naturality of the Pimsner-Voiculescu sequence and the Five Lemma yields that $K_*(\eta \rtimes \Z^k)$ is an isomorphism. If $k = \infty$, it follows from continuity of $K$-theory that $K_*(\eta \rtimes \Z^\infty)$ is an isomorphism, since $K_*(\eta \rtimes \Z^k)$ is an isomorphism for every $k\in \N$.

Let $u \colon \Z^k \to \U(\M((A \oplus B) \rtimes_{\alpha \oplus \beta} \Z^k))$ and $\tilde{u} \colon \Z^k \to \U(\M(M_2(C)\rtimes_\delta \Z^k))$ denote the canonical representations, respectively. The covariant pair given by the natural inclusion  $A \into 1_{\M(A)}((A \oplus B) \rtimes_{\alpha \oplus \beta} \Z^k) 1_{\M(A)}$ and the unitary representation $1_{\M(A)}u_h$, $h \in \Z^k$, gives rise to a $*$-homomorphism $\Phi_A \colon A \rtimes_\alpha \Z^k \to (A \oplus B) \rtimes_{\alpha \oplus \beta} \Z^k$. Similarly, we define $\Phi_B \colon B \rtimes_\beta \Z^k \to (A \oplus B) \rtimes_{\alpha \oplus \beta} \Z^k$. It is easy to check that $\Phi_A$ and $\Phi_B$ are orthogonal and 
\[
\Phi_A \oplus \Phi_B \colon A \rtimes_\alpha \Z^k \oplus B \rtimes_\beta \Z^k \to (A \oplus B) \rtimes_{\alpha \oplus \beta} \Z^k
\]
is an isomorphism. Moreover, let $\tilde{\varphi} \colon A \rtimes_\alpha \Z^k \to M_2(C) \rtimes_\delta \Z^k$ be the $*$-homomorphism induced by the covariant pair in $\M(e_{11}(M_2(C)\rtimes_\delta \Z^k) e_{11})$ given by the composition of the embedding $C \into M_2(C)$ into the upper left corner with $\varphi$ and the unitary representation $e_{11}\tilde{u}_h$, $h \in \Z^k$. Define $\tilde{\psi} \colon B \rtimes_\beta \Z^k \to M_2(C) \rtimes_\delta \Z^k$ analogously by considering the embedding $C \into M_2(C)$ into the lower right corner. By construction, the following diagram commutes
\[
\xymatrix{
A \rtimes_\alpha \Z^k \oplus B \rtimes_\beta \Z^k \ar[drr]_{\tilde{\varphi} \oplus \tilde{\psi}} \ar[rr]_\cong^{\Phi_A \oplus \Phi_B} & & (A \oplus B) \rtimes_{\alpha \oplus \beta} \Z^k \ar[d]^{\eta \rtimes \Z^k} \\
& & M_2(C) \rtimes_\delta \Z^k
}\]
which shows that
\[
K_*(\tilde{\varphi}) \oplus K_*(\tilde{\psi})\colon  K_*(A \rtimes_\alpha \Z^k) \oplus K_*(B \rtimes_\beta \Z^k) \to K_*(M_2(C) \rtimes_{\delta} \Z^k)
\]
is an isomorphism.

Let $\kappa'\colon M_2(C) \rtimes_{\delta} \Z^k \stackrel{\cong}{\longrightarrow} M_2(C) \rtimes_{\gamma^{(2)}} \Z^k$ denote the isomorphism induced by the $\gamma^{(2)}$-cocycle $w$. Then the following diagram commutes and the proof is complete:
\[
\xymatrix{
 A \rtimes_\alpha \Z^k \ar[d]_{\tilde{\varphi}} \ar[r]^{\varphi \rtimes \Z^k} & C \rtimes_\gamma \Z^k  \ar[rd]^{\id_{C \rtimes_\gamma \Z^k} \oplus 0} \\
 M_2(C) \rtimes_\delta \Z^k \ar[r]^{\kappa'}_\cong & M_2(C) \rtimes_{\gamma^{(2)}} \Z^k \ar[r]^\cong & M_2(C \rtimes_\gamma \Z^k) \\
 B \rtimes_\beta \Z^k \ar[u]^{\tilde{\psi}} \ar[r]^*!/_0.5mm/{\labelstyle \psi \rtimes \Z^k} & C \rtimes_{\tilde{\gamma}} \Z^k \ar[r]^\kappa_\cong & C \rtimes_\gamma \Z^k \ar[u]_*!/_1.5mm/{\labelstyle 0 \oplus \id_{C\rtimes_\gamma \Z^k}}
}
\]
\end{proof}

\begin{rem}
Proposition~\ref{prop:splitted K-theory} is true in a more general setting. In fact, $\Z^k$ could be replaced by any locally compact group $G$ with the following property: If $\varphi\colon  A \to B$ is an $\alpha$ - $\beta$-equivariant $*$-homomorphism such that $K_*(\varphi)$ is an isomorphism, then $K_*(\varphi \rtimes G)$ is an isomorphism as well.
\end{rem}

\begin{rem}\label{rem:dilations}
Note that $K_1(M_{d^\infty}) = 0$ and the natural embedding $j\colon M_{d^\infty} \into \BB_S$ induces an isomorphism between the corresponding $K_0$-groups. The invariance of $M_{d^\infty} \subset \BB_S$ under the $H^+$-action $\alpha$, see Remark~\ref{rem:BD subalg B_S}, yields a non-degenerate $*$-homomorphism $j_\infty\colon M_{d^\infty,\infty} \to \BB_{S,\infty}$ between the minimal automorphic dilations for $\alpha$, which is equivariant for the induced $H$-actions $\alpha_\infty$. From the concrete model of the minimal automorphic dilation as an inductive limit, see \cite{Lac}*{Proof of Theorem~2.1}, we conclude that $K_1(M_{d^\infty,\infty}) = 0$ and $K_0(j_\infty)$ is an isomorphism. Moreover, there is an isomorphism between $\BB_S$ and $C(\Delta)\rtimes \Z$ that intertwines the actions of $H^+$, see Remark~\ref{rem:spec of D_S}. It then follows from \eqref{eq:full corner B_S} and \cite{Lac}*{Theorem~2.1} that $\alpha\colon H^+ \curvearrowright C(\Delta)\rtimes\Z$ dilates to $\widetilde{\alpha}\colon H \curvearrowright C_0(\Omega)\rtimes N$, where $\widetilde{\alpha}$ coincides with the $H$-action from Lemma~\ref{lem:torsion algebra Morita}. Consequently, there is an $\alpha_\infty$ - $\widetilde{\alpha}$-equivariant isomorphism $\BB_{S,\infty} \stackrel{\cong}{\longrightarrow} C_0(\Omega)\rtimes N$.
\end{rem}

As in Section~\ref{sec:comp with real dyn}, let $\iota_N\colon C_0(\R) \into C_0(\R)\rtimes N$ denote the canonical embedding. Note that $\iota_N$ is non-degenerate and equivariant with respect to the $H$-actions $\beta$ (and also $\beta^{-1}$).

\begin{thm}\label{thm:decomposition of K-theory}
The map
\[
K_*(j \rtimes^e H^+) \oplus K_*(\iota_N \rtimes H)\colon
K_*(M_{d^\infty} \rtimes^e_\alpha H^+) \oplus K_*(C_0(\R) \rtimes_{\beta} H) \to K_*(\QQ_S),
\]
induced by the identifications $\BB_S \rtimes^e_\alpha H^+ \cong \QQ_S$ from \eqref{eq:B_S and Q_S as crossed products} and $(C_0(\R) \rtimes N) \rtimes_\beta H \cong \QQ_S \otimes \K$ from Theorem~\ref{thm:duality}, is an isomorphism.
\end{thm}
\begin{proof}
By combining Remark~\ref{rem:dilations} with Lemma~\ref{lem:torsion algebra Morita}, there exist an $H$-action $\widetilde\beta$ on $C_0(\R) \rtimes N$ that is exterior equivalent to $\beta^{-1}$, and a non-degenerate $\alpha_\infty$ - $\widetilde\beta$-equivariant $*$-homomorphism $\psi\colon M_{d^\infty,\infty}\to C_0(\R) \rtimes N$, namely the one coming from the composition
\[
M_{d^\infty,\infty} \stackrel{j_\infty}{\longrightarrow}
\BB_{S,\infty} \stackrel{\cong}{\longrightarrow}
C_0(\Omega) \rtimes N \stackrel{\cong}{\longrightarrow}
C_0(\R) \rtimes N.
\]
Since $K_1(j_\infty) = 0$ and $K_0(j_\infty)$ is an isomorphism by Remark~\ref{rem:dilations}, the same also holds for $K_1(\psi)$ and $K_0(\psi)$, respectively. Now Proposition~\ref{prop:KtheoryA0} gives that $K_0(\iota_N)$ is trivial and $K_1(\iota_N)$ is an isomorphism. As $\psi$ and $\iota_N$ are non-degenerate,
\begin{multline*}
K_*(\kappa \circ (\psi \rtimes H)) \oplus K_*(\iota_N \rtimes H)\colon \\
K_*(M_{d^\infty,\infty} \rtimes_{\alpha_\infty} H) \oplus K_*(C_0(\R) \rtimes_{\beta^{-1}} H) \to K_*((C_0(\R)\rtimes N)\rtimes_{\beta^{-1}} H)
\end{multline*}
is an isomorphism by Proposition~\ref{prop:splitted K-theory}, where $\kappa\colon(C_0(\R)\rtimes N)\rtimes_{\widetilde\beta} H \stackrel{\cong}{\longrightarrow} (C_0(\R)\rtimes N) \rtimes_{\beta^{-1}} H$ denotes the isomorphism induced by a fixed $\beta^{-1}$-cocycle defining $\widetilde{\beta}$. Since $K_*(j \rtimes^e H^+)$ corresponds to $K_*(j_\infty \rtimes H)$ under the isomorphisms induced by the minimal automorphic dilations, we also get that $K_*(j \rtimes^e H^+)$ corresponds to $K_*(\kappa \circ (\psi \rtimes H))$ under the isomorphism $K_*(\BB_S \rtimes^e_\alpha H^+) \cong K_*((C_0(\R) \rtimes N) \rtimes_{\beta^{-1}} H)$. As $\BB_S \rtimes^e_\alpha H^+ \cong \QQ_S$ by \eqref{eq:B_S and Q_S as crossed products} and $(C_0(\R)\rtimes N)\rtimes_{\beta^{-1}} H \cong (C_0(\R)\rtimes N)\rtimes_\beta H \cong \QQ_S \otimes \mathcal{K}$ by Theorem~\ref{thm:duality}, the conclusion follows.
\end{proof}

We will now show that the two summands appearing in Theorem~\ref{thm:decomposition of K-theory} correspond to the torsion and the free part of $K_*(\QQ_S)$, respectively.

\begin{prop}\label{prop:K-theory torsion free part}
For $i=0,1$, $K_i(C_0(\R) \rtimes_{\beta} H)$ is the free abelian group in $2^{\lvert S \rvert-1}$ generators.
\end{prop}
\begin{proof}
The result holds for any non-trivial subgroup $H$ of $\Q^\times_+$, and we prove it in generality, not necessarily requiring $H$ to be generated by $S$. Suppose that $k$ is the (possibly infinite) rank of $H$. Let $\{ h_i : 0 \leq i \leq k \}$ be a minimal generating set for $H$. For $t\in [0,1]$ and $1 \leq i\leq k$, define $\tilde{\beta}_{h_i,t} \in \op{Aut}(C_0(\R))$ by $\tilde{\beta}_{h_i,t}(f)(s) = f((th_i^{-1} + 1 - t)s)$. Note that $\tilde{\beta}_{h_i,t}$ is indeed an automorphism as $h_i > 0$. Since multiplication on $\R$ is commutative, we see that for each $t \in [0,1]$, $\left\lbrace \tilde{\beta}_{h_i,t} \right\rbrace_{1\leq i\leq k}$ defines an $H$-action. Let $\gamma\colon  H \curvearrowright C_0([0,1],C_0(\R))$ be the action given by $\gamma_{h_i}(f)(t) = \tilde{\beta}_{h_i,t}(f(t))$. We have the following short exact sequence of $C^*$-algebras:
\[
C_0((0,1],C_0(\R))\rtimes_\gamma H \into C_0([0,1],C_0(\R))\rtimes_\gamma H \stackrel{\op{ev}_0\rtimes H}{\onto} C_0(\R)\rtimes_{\id} H
\]
The Pimsner-Voiculescu sequence shows that $K_*(C_0((0,1],C_0(\R))\rtimes_\gamma H) = 0$, where we also use  continuity of $K$-theory if $k = \infty$. The six-term exact sequence corresponding to the above extension now yields that $K_*(\op{ev}_0 \rtimes H)$ is an isomorphism. A similar argument shows that $K_*(\op{ev}_1 \rtimes H)$ is an isomorphism. We therefore conclude that for $i=0,1$,
\[K_i(C_0(\R) \rtimes_{\beta} H) \cong K_i(C_0(\R) \rtimes_{\id} H) \cong K_i(C_0(\R) \otimes C^*(H)).\]
This completes the proof as $K_i(C_0(\R) \otimes C^*(H))$ is the free abelian group in $2^{k-1}$ generators.
\end{proof}

\begin{prop} \label{prop:UHF cross prod torsion K-th}
$K_*(M_{d^\infty} \rtimes^e_\alpha H^+)$ is a torsion group, which is finite if $S$ is finite.
\end{prop}
\begin{proof}
As in Remark \ref{rem:BD subalg B_S}, we think of $M_{d^\infty}$ as the inductive limit $(M_p(\C),\iota_{p,pq})_{p,q \in H^+}$ with $\iota_{p,pq}\colon M_p(\C) \to M_{pq}(\C)$ given by $e_{i, j}^{(p)} \otimes 1 \mapsto \sum_{k=0}^{q-1}e_{i+pk, j+pk}^{(pq)} \otimes 1$. With this perspective, $\alpha$ satisfies $\alpha_q(e^{(p)}_{m,m}) = e^{(pq)}_{qm,qm}$ for all $p,q \in H^+$ and $0 \leq m \leq p-1$. From this, one concludes that for $q \in H^+$, $K_0(\alpha_q)$ is given by multiplication with $1/q$ on $K_0(M_{d^\infty}) \cong N$. Hence, for $p \in S$, there exists a Pimsner-Voiculescu type exact sequence, see \cite{Pas}*{Theorem~4.1} and also \cite{CunPV}*{Proof of Proposition~3.1},
\[
\xymatrix{
0 \ar[r] & K_1(M_{d^\infty} \rtimes^e_{\alpha_p} \N) \ar[r] & N \ar[r]^{\frac{p-1}{p}} & N \ar[r] & K_0(M_{d^\infty} \rtimes^e_{\alpha_p} \N) \ar[r] & 0
}
\]
This shows that $K_1(M_{d^\infty} \rtimes^e_{\alpha_p} \N) = 0$ and $K_0(M_{d^\infty} \rtimes^e_{\alpha_p} \N) \cong N / (p-1)N$. In particular, $K_*(M_{d^\infty} \rtimes^e_{\alpha_p} \N)$ is a torsion group. If $S$ is finite, we can write $M_{d^\infty} \rtimes^e_\alpha H^+$ as an $\lvert S\rvert$-fold iterative crossed product by $\N$ and apply the Pimsner-Voiculescu type sequence repeatedly to get that $K_*(M_{d^\infty} \rtimes^e_\alpha H^+)$ is a torsion group. If $S$ is infinite, we may use continuity of $K$-theory to conclude the claim from the case of finite $S$.

Finiteness of $S$ implies finiteness of $K_*(M_{d^\infty} \rtimes^e_\alpha H^+)$ because $N/(p-1)N$ is finite for all $p \in S$, which follows from the forthcoming Lemma~\ref{lem:N/gN}.
\end{proof}

Using Proposition~\ref{prop:K-theory torsion free part} and \ref{prop:UHF cross prod torsion K-th}, we record the following immediate consequence of the decomposition of $K_*(\QQ_S)$ given in Theorem~\ref{thm:decomposition of K-theory}.

\begin{cor}\label{cor:torsion and free part K-theory}
$K_*(\QQ_S)$ decomposes as a direct sum of a free abelian group and a torsion group. More precisely, $K_*(j \rtimes^e H^+)$ is a split-injection onto the torsion subgroup and $K_*(\iota_N \rtimes H)$ is a split-injection onto the torsion free part of $K_*(\QQ_S)$, respectively.
\end{cor}

\section{The torsion subalgebra}\label{sec:torsion part}
Within this section we analyze the structure of $M_{d^\infty} \rtimes^e_\alpha H^+$ and its role relative to $\QQ_S$ more closely. First, we show that the inclusion $M_{d^\infty} \into \BB_S$ is equivariantly sequentially split with respect to the $H^+$-actions $\alpha$ in the sense of \cite{BarSza1}*{Remark~3.17}, see Proposition~\ref{prop:UHF into BD is eq seq split}. According to \cite{BarSza1}, we thus get that $M_{d^\infty} \rtimes^e_\alpha H^+$ shares many structural properties with $\BB_S \rtimes^e_\alpha H^+ \cong \QQ_S$. Most importantly, $M_{d^\infty} \rtimes^e_\alpha H^+$ is a unital UCT Kirchberg algebra, see Corollary~\ref{cor:UHF into BD seq split cr pr}. By simplicity of $M_{d^\infty} \rtimes^e_\alpha H^+$, we conclude that this $C^*$-algebra is in fact isomorphic to the natural subalgebra $\CA_S$ of $\QQ_S$ that is generated by all the isometries $u^ms_p$ with $p \in S$ and $0 \leq m \leq p-1$, see Corollary~\ref{cor:subalgebra for torsion part}. By Corollary \ref{cor:torsion and free part K-theory}, it thus follows that the canonical inclusion $\CA_S \into \QQ_S$ induces a split-injection onto the torsion subgroup of $K_*(\QQ_S)$. Due to this remarkable feature, we call $\CA_S$ the \emph{torsion subalgebra} of $\QQ_S$. 

We then present two additional interesting perspectives on the torsion subalgebra $\CA_S$. Firstly, $\CA_S$ can be described as the boundary quotient of the right LCM subsemigroup $U = \{ (m,p) : p \in H^+, 0 \leq m \leq p-1\}$ of $\N \rtimes H^+$ in the sense of \cite{BRRW}, see Proposition~\ref{prop:A_S as BQ of U}. This yields a commutative diagram which might be of independent interest, see Remark~\ref{rem:A_S as BQ of U}.

Secondly, the boundary quotient perspective allows us to identify $\CA_S$ for $k:= \lvert S \rvert < \infty$ with the $C^*$-algebra of the $k$-graph $\Lambda_{S,\theta}$ consisting of a single vertex with $p$ loops of color $p$ for every $p \in S$, see Corollary~\ref{cor:tor subalgebra via k-graphs}. Quite intriguingly, $\Lambda_{S,\theta}$ differs from the canonical $k$-graph model $\Lambda_{S,\sigma}$ for $\bigotimes_{p \in S} \CO_p$ only with respect to its factorization rules, see Remark~\ref{rem:Lambda_S flip}. In fact, the corresponding $C^*$-algebras coincide for $\lvert S \rvert \leq 2$, see Proposition~\ref{prop:A_S for |S|=2}. After obtaining these intermediate results, we were glad to learn from Aidan Sims that, in view of Conjecture~\ref{conj:k-graph}, it is reasonable to expect that the results for $\lvert S \rvert \leq 2$ already display the general form, i.e.\ that $\CA_S$ is always isomorphic to $\bigotimes_{p \in S} \CO_p$.

\begin{prop}\label{prop:UHF into BD is eq seq split}
The embedding $M_{d^\infty} \into \BB_S$ is $\alpha$-equivariantly sequentially split.
\end{prop}
\begin{proof}
Let $\iota\colon M_{d^\infty} \into \prod_{p \in H^+}M_{d^\infty} \big / \bigoplus_{p \in H^+} M_{d^\infty}$ denote the canonical inclusion as constant sequences and $\bar{\alpha}$ the induced action of $H^+$ on $\prod_{p \in H^+}M_{d^\infty} \big / \bigoplus_{p \in H^+} M_{d^\infty}$ given by componentwise application of $\alpha_h$ for $h \in H$. Clearly, $\prod_{p \in H^+}M_{d^\infty} \big / \bigoplus_{p \in H^+} M_{d^\infty}$ is canonically isomorphic to the sequence algebra of $M_{d^\infty}$, $\prod_{n \in \N}M_{d^\infty} \big / \bigoplus_{n \in \N} M_{d^\infty}$. In particular, this isomorphism intertwines $\bar{\alpha}$ and the natural $H^+$-action on the sequence algebra induced by $\alpha$. We therefore need to construct a $\alpha$ - $\bar{\alpha}$-equivariant $*$-homomorphism $\chi\colon \mathcal{B}_S \to \prod_{p \in H^+}M_{d^\infty} \big / \bigoplus_{p \in H^+} M_{d^\infty}$ making the following diagram commute: 
\begin{equation}\label{dia:UHF into BD is eq seq split}
\begin{gathered}
\begin{xy}
\xymatrix{
M_{d^\infty} \ar@<-1ex>@{^{(}->}[dr] \ar[rr]^(0.35){\iota}&&\prod\limits_{p \in H^+}M_{d^\infty} \big / \bigoplus\limits_{p \in H^+} M_{d^\infty}\\
&\mathcal{B}_S \ar@<-1ex>@{-->}_(0.35){\chi}[ur]}
\end{xy}
\end{gathered}
\end{equation}
Recall the inductive system $(M_p(\C) \otimes C^*(\Z),\iota_{p,pq})_{p,q \in H^+}$ from Remark~\ref{rem:BD subalg B_S} whose inductive limit is isomorphic to $\BB_S$. The canonical subalgebra $M_p(\C) \subset M_p(\C) \otimes C^*(\Z)$ can in this way be considered as a subalgebra of $M_{d^\infty} \subset \BB_S$ in a natural way. For each $p \in H^+$, the map $\chi_p\colon  M_p(\C) \otimes C^*(\Z) \to M_p(\C)$ given by $\sum_{k=1}^n a_k \otimes u \mapsto \sum_{k=1}^n a_k$ is a $*$-homomorphism. Thus, the family $(\chi_p)_{p \in H^+}$ gives rise to a $*$-homomorphism 
\[
\begin{array}{c} 
\chi'\colon  \prod\limits_{p \in H^+} M_p(\C) \otimes C^*(\Z) \to \prod\limits_{p \in H^+} M_{d^\infty}.\end{array}
\]
Clearly, $\chi'\bigl(\bigoplus_{p \in H^+} M_p(\C) \otimes C^*(\Z)\bigr) \subset \bigoplus_{p \in H^+} M_{d^\infty}$, so $\chi'$ induces a map
\[
\begin{array}{c} \chi\colon  \prod\limits_{p \in H^+} M_p(\C) \otimes C^*(\Z) \bigr/ \bigl(\bigoplus\limits_{p \in H^+} M_p(\C) \otimes C^*(\Z)\bigr) \to \prod\limits_{p \in H^+}M_{d^\infty} \big / \bigoplus\limits_{p \in H^+} M_{d^\infty}. \end{array}
\] 
Using the inductive limit description of $\BB_S$ from Remark~\ref{rem:BD subalg B_S}, we can think of $\BB_S$ as a subalgebra of $\prod_{p \in H^+} M_p(\C) \otimes C^*(\Z) \bigr/ \bigl(\bigoplus_{p \in H^+} M_p(\C) \otimes C^*(\Z)\bigr)$. Moreover, because of the concrete realization of $M_{d^\infty}$ as the inductive limit associated with $(M_p(\C),\iota_{p,pq})_{p,q \in H^+}$, we have that $\chi$ restricts to the canonical embedding $\iota$ on $M_{d^\infty}$. Hence, \eqref{dia:UHF into BD is eq seq split} is commutative, when we ignore the question of equivariance, or, in other words, $\iota$ is sequentially split as an ordinary $*$-homomorphism. However, we claim that we also have a commutative diagram
\begin{equation}\label{dia:UHF into BD equivariance}
\begin{gathered}
\begin{xy}
\xymatrix{
\prod\limits_{p \in H^+}M_{d^\infty} \big / \bigoplus\limits_{p \in H^+} M_{d^\infty} \ar^{\bar{\alpha}_p}[r] & \prod\limits_{p \in H^+}M_{d^\infty} \big / \bigoplus\limits_{p \in H^+} M_{d^\infty} \\
\mathcal{B}_S \ar[u]^(0.4){\chi} \ar_{\alpha_p}[r] & \mathcal{B}_S \ar[u]_(0.4){\chi}
}
\end{xy}
\end{gathered}
\end{equation}
for each $p \in H^+$. Let us expand this diagram for fixed $p$ and arbitrary $q \in H^+$ to:
\begin{equation}\label{dia:UHF into BD equivariance zoom}
\begin{gathered}
\scalebox{0.8}{\begin{xy}
\xymatrix{
\prod\limits_{p \in H^+}M_{d^\infty} \big / \bigoplus\limits_{p \in H^+} M_{d^\infty} \ar^{\bar{\alpha}_p}[rrr] &&& \prod\limits_{p \in H^+}M_{d^\infty} \big / \bigoplus\limits_{p \in H^+} M_{d^\infty} \\
&M_q(\C) \ar@{_{(}->}[ul] \ar^{\alpha_p}[r] & M_{pq}(\C) \ar@{^{(}->}[ur] \\
&M_q(\C) \otimes C^*(\Z) \ar@{^{(}->}[dl] \ar^{\chi_q}[u] \ar_*!/_0.5mm/{\labelstyle \alpha_p}[r] & M_{pq} \otimes C^*(\Z) \ar_{\chi_{pq}}[u] \ar@{_{(}->}[dr]\\
\mathcal{B}_S \ar[uuu]^{\chi} \ar_*!/_0.5mm/{\labelstyle \alpha_p}[rrr] &&& \mathcal{B}_S \ar[uuu]_{\chi}
}
\end{xy}}
\end{gathered}
\end{equation}
It is clear that the four outer chambers are commutative, so we only need to check the centre. For every $0 \leq i,j \leq q-1$, we get
\[\chi_{pq} \circ \alpha_p (e_{i, j}^{(q)} \otimes u) = \chi_{pq}(e_{pi, pj}^{(pq)} \otimes u^p) = e_{pi, pj}^{(pq)} = \alpha_p \circ \chi_q (e_{i, j}^{(q)} \otimes u)\]
and therefore $\chi_{pq} \circ \alpha_p = \alpha_p \circ \chi_q$ on $M_q(\C) \otimes C^*(\Z)$. This establishes the claim as we have $M_{d^\infty} = \varinjlim(M_q(\C),q\in H^+)$ and $\BB_S = \varinjlim(M_q(\C) \otimes C^*(\Z),q\in H^+)$.
\end{proof}

\begin{cor}\label{cor:UHF into BD seq split cr pr}
The inclusion $M_{d^\infty} \rtimes^e_\alpha H^+ \to \BB_S \rtimes^e_\alpha H^+$ is sequentially split. In particular, $M_{d^\infty} \rtimes^e_\alpha H^+$ is a UCT Kirchberg algebra.
\end{cor}
\begin{proof}
By Proposition~\ref{prop:UHF into BD is eq seq split}, we know that $M_{d^\infty} \into \BB_S$ is $\alpha$-equivariantly sequentially split. As this inclusion preserves the units, we can use the universal property of the semigroup crossed products $M_{d^\infty} \rtimes^e_\alpha H^+$ and $\BB_S \rtimes^e_\alpha H^+$ to obtain a commutative diagram of $*$-homomorphisms
\[
\xymatrix{
M_{d^\infty} \rtimes^e_\alpha H^+ \ar[dr] \ar[rr]^(0.4){\iota \rtimes^e H^+} && \left(\prod\limits_{p \in H^+} M_{d^\infty} \big / \bigoplus\limits_{p \in H^+} M_{d^\infty} \right)\rtimes^e_{\bar{\alpha}} H^+\\
&\BB_S \rtimes^e_\alpha H^+ \ar[ur]
}
\]
Again by the universal property of semigroup crossed products, there is a natural $*$-homomorphism
\[\begin{array}{c}
\psi\colon\left(\prod\limits_{p \in H^+} M_{d^\infty} \big / \bigoplus\limits_{p \in H^+} M_{d^\infty} \right)\rtimes^e_{\bar{\alpha}} H^+ \to \prod\limits_{p \in H^+} M_{d^\infty} \rtimes^e_\alpha H^+ \big /\bigoplus\limits_{p \in H^+} M_{d^\infty} \rtimes^e_\alpha H^+
\end{array}\]
such that $\psi \circ (\iota \rtimes^e H^+)$ coincides with the standard embedding. This shows that the inclusion $M_{d^\infty} \rtimes^e_\alpha H^+ \to \BB_S \rtimes^e_\alpha H^+$ is sequentially split. It now follows from \cite{BarSza1}*{Theorem~2.9~(1)+(8)} that $M_{d^\infty} \rtimes^e_\alpha H^+$ is a Kirchberg algebra. Moreover, $M_{d^\infty} \rtimes^e_\alpha H^+$ satisfies the UCT by \cite{BarSza1}*{Theorem~2.10}. We note that this part also follows from standard techniques combined with the central result of \cite{Lac}.
\end{proof}

We will now see that simplicity enables us to identify $M_{d^\infty} \rtimes^e_\alpha H^+$ with the following natural subalgebra of $\QQ_S$, whose name is justified by the next result.

\begin{defn}\label{def:torsion subalgebra}
The \emph{torsion subalgebra} $\CA_S$ of $\QQ_S$ is the $C^*$-subalgebra of $\QQ_S$ generated by $\{u^ms_p : p \in S, 0 \leq m \leq p-1\}$.
\end{defn}  

Note that for $S=\{p\}$, the subalgebra $\CA_S$ is canonically isomorphic to $\CO_p$.

\begin{cor}\label{cor:subalgebra for torsion part}\label{cor:torsion subalgebra justification}
The isomorphism $\BB_S \rtimes^e_\alpha H^+ \stackrel{\cong}{\longrightarrow} \QQ_S$ from \eqref{eq:B_S and Q_S as crossed products} restricts to an isomorphism $M_{d^\infty} \rtimes^e_\alpha H^+ \stackrel{\cong}{\longrightarrow} \CA_S$. In particular, the canonical inclusion $\CA_S \into \QQ_S$ induces a split-injection onto the torsion subgroup of $K_*(\QQ_S)$.
\end{cor}
\begin{proof}
$\CA_S$ contains the copy of $M_{d^\infty} \subset \mathcal{B}_S$ described in Remark~\ref{rem:BD subalg B_S}. Together with $s_p, p \in S$, which are also contained in $\CA_S$, this  defines a covariant representation of $(M_{d^\infty},\alpha)$ inside $\CA_S$. The resulting $*$-homomorphism $M_{d^\infty} \rtimes^e_\alpha H^+ \to \CA_S$ is surjective. By Corollary~\ref{cor:UHF into BD seq split cr pr}, $M_{d^\infty} \rtimes^e_\alpha H^+$ is simple, so this map is an isomorphism. The second claim is due to Corollary~\ref{cor:torsion and free part K-theory}.
\end{proof}

Let us continue with the representation of $\CA_S$ as a boundary quotient. When $\lvert S\rvert=k<\infty$, this will lead us to a $k$-graph model for $\CA_S$ that is closely related to the canonical $k$-graph representation for $\bigotimes_{p \in S}\CO_p$, see Remark~\ref{rem:Lambda_S flip}. Consider the subsemigroup $U:= \{ (m,h) \in \N \rtimes H^+ : 0 \leq m \leq h-1\}$ of $\N \rtimes H^+$. Observe that $U$ is a right LCM semigroup because 
\begin{equation}\label{eq:right LCM subsemigroup}
(m,h)U \cap (m',h')U = \bigl((m,h)(\N \rtimes H^+) \cap (m',h')(\N \rtimes H^+)\bigr) \cap U 
\end{equation}
for all $(m,h),(m',h') \in U$, and $\N \rtimes H^+$ is right LCM. We note that $U$ can be used to describe $\N \rtimes H^+$ as a Zappa-Sz\'{e}p product $U \bowtie \N$, where action and restriction are given in terms of the generator $1 \in \N$ and $(m,h) \in U$ by 
\[1.(m,h) = \begin{cases} (m+1,h) &\text{if } m<h-1, \\ (0,h) &\text{if } m=h-1,\end{cases} \quad \quad 1\rvert_{(m,h)} = \begin{cases} 0 &\text{if } m<h-1, \text{ and}\\ 1 &\text{if } m=h-1.\end{cases}\]
In the case of $H^+=\N^\times$ this has been discussed in detail in \cite{BRRW}*{Subsection~3.2} and the very same arguments apply for the cases we consider here. 

\begin{prop}\label{prop:A_S as BQ of U}
$\CA_S$ is canonically isomorphic to the boundary quotient $\QQ(U)$.
\end{prop}
\begin{proof}
Recall that $\QQ(U)$ is the quotient of the full semigroup $C^*$-algebra $C^*(U)$ by relation \eqref{eq:BQ}. In particular, it is generated as a $C^*$-algebra by a representation $v$ of $U$ by isometries whose range projections are denoted $v^{\phantom{*}}_{(m,h)}v_{(m,h)}^*= e^{\phantom{*}}_{(m,h)U}$ for $(m,h) \in U$.

For every $h \in H^+$, we get a family of matrix units $(v^{\phantom{*}}_{(m,h)}v_{(n,h)}^*)_{0 \leq m,n \leq h-1}$ because
\[\begin{array}{l}
v_{(n,h)}^*v^{\phantom{*}}_{(m,h)} =  v_{(n,h)}^*e^{\phantom{*}}_{(n,h)U \cap (m,h)U}v_{(m,h)} = \delta_{m,n} \quad
\text{and} \quad \sum\limits_{m=0}^{h-1} e_{(m,h)U} = 1
\end{array}\] 
as $\{(m,h) : 0 \leq m \leq h-1\}$ is an accurate foundation set for $U$. That is to say that, for each $u \in U$, there is $0 \leq m \leq h-1$ such that $uU \cap (m,h)U \neq \emptyset$, and $(m,h)U \cap (n,h)U = \emptyset$ unless $m=n$, see \cite{bsBQforADS} for further details. Since 
\[\begin{array}{lcl}
v^{\phantom{*}}_{(m,h)}v_{(n,h)}^* &=& v^{\phantom{*}}_{(m,h)} \bigl(\sum\limits_{k=0}^{h'-1}e^{\phantom{*}}_{(k,h')U}\bigr)v_{(n,h)}^*\\
&=& \sum\limits_{k=0}^{h'-1}v^{\phantom{*}}_{(m+hk,hh')}v_{(n+hk,hh')}^*
\end{array}\]
for each $h' \in H^+$, we see that $C^*(\{v^{\phantom{*}}_{(m,h)}v_{(n,h)}^* : h \in H^+, 0 \leq m,n \leq h-1\}) \subset \QQ(U)$ is isomorphic to $M_{d^\infty}$. In fact, we get a covariant representation for $(M_{d^\infty},H^+,\alpha)$ as 
\[v^{\phantom{*}}_{(0,p)}v^{\phantom{*}}_{(m,h)}v_{(n,h)}^*v_{(0,p)}^* = v^{\phantom{*}}_{(pm,ph)}v_{(pn,ph)}^*.\] 
Thus we get a $*$-homomorphism $\varphi\colon  \CA_S \cong M_{d^\infty} \rtimes^e_\alpha H^+ \to \QQ(U)$ given by $u^ms_h \mapsto v_{(m,h)}$, see Corollary~\ref{cor:subalgebra for torsion part}. The map is surjective, and due to Corollary~\ref{cor:UHF into BD seq split cr pr}, the domain is simple so that $\varphi$ is an isomorphism.
\end{proof}

\begin{rem}\label{rem:A_S as BQ of U}
Conceptually, it seems that there is more to Proposition~\ref{prop:A_S as BQ of U} than the proof entails: There is a commutative diagram 
\begin{equation}\label{eq:right LCM inclusion diagram}
\begin{gathered}
\xymatrix{
C^*(U) \ar^(0.4){\iota}@{^{(}->}[r] \ar@{->>}_{\pi_U}[d] & C^*(\N \rtimes H^+) \ar@{->>}^{\pi_{\N \rtimes H^+}}[d] \\
\QQ(U) \ar@{^{(}->}_(0.4){\varphi^{-1}}[r] & \QQ(\N \rtimes H^+)
}
\end{gathered}
\end{equation}
with $\iota$ induced by $U \subset \N \rtimes H^+$ and $\varphi$ as in the proof of Proposition~\ref{prop:A_S as BQ of U}. The fact that $\iota$ is an injective $*$-homomorphism follows from \cite{BLS2}*{Proposition~3.6}: $N \rtimes H$ is amenable and hence $C^*(\N \rtimes H^+) \cong C^*_r(\N \rtimes H^+)$, see \cite{BLS1}*{Example~6.3}, and similarly $C^*(U) \cong C^*_r(U)$. Note that the bottom row of \eqref{eq:right LCM inclusion diagram} is given by $\CA_S \into \QQ_S$, see Proposition~\ref{prop:Q_S as BQ} and Proposition~\ref{prop:A_S as BQ of U}.
\end{rem}

By Corollary~\ref{cor:torsion subalgebra justification} and Proposition~\ref{prop:A_S as BQ of U}, the torsion part of the $K$-theory of the boundary quotient of $\N \rtimes H^+$ arises from the boundary quotient of the distinguished submonoid $U$, which in fact sits inside $\QQ(\N \rtimes H^+)$ in the natural way. 

For the remainder of this section, we will assume that $S$ is finite with cardinality $k$. This restriction is necessary in order to derive a $k$-graph model for $\CA_S$, which we obtain via the boundary quotient representation of $\CA_S$. Note that for $p,q \in \N^\times$ and $(m,n) \in \{0,\dots,p-1\} \times \{0,\dots,q-1\}$, there is a unique pair $(n',m') \in \{0,\dots,q-1\} \times \{0,\dots,p-1\}$ such that $m+pn = n'+qm'$. In other words, the map 
\[\theta_{p,q}\colon  \{0,\dots,p-1\} \times \{0,\dots,q-1\} \to \{0,\dots,q-1\} \times \{0,\dots,p-1\}\] 
with $ (m,n) \mapsto (n',m')$ determined by $n'+qm' = m+pn$ is bijective. 

\begin{rem}\label{rem:k-graph from S}
For each $p \in S$, we can consider the $1$-graph given by a single vertex with $p$ loops $(m,p), 0 \leq m \leq p-1$. If we think of the collection of these $1$-graphs as the skeleton of a $k$-graph, i.e.\ the set of all edges of length at most $1$, where the vertices for different $p$ are identified, then the maps $\theta_{p,q}$ satisfy condition~(2.8) in \cite{FS}*{Remark~2.3}, and hence define a row-finite $k$-graph $\Lambda_{S,\theta}$. Indeed, this is obvious for $k=2$. For $k \geq 3$, let $p,q,r \in S$ be pairwise distinct elements and fix $0 \leq m_t \leq t-1$ for $t = p,q,r$. We compute
\[\begin{array}{lclcl}
m_p+p(m_q+qm_r) &=& m_p+p(m^{(1)}_r+rm^{(1)}_q) &=& m^{(2)}_r+r(m^{(1)}_p+pm^{(1)}_q)\\
&=& m^{(2)}_r+r(m^{(2)}_q+qm^{(2)}_p) &=& m^{(3)}_q+q(m^{(3)}_r+rm^{(2)}_p)\\
&=& m^{(3)}_q+q(m^{(3)}_p+pm^{(4)}_r) &=& m^{(4)}_p+p(m^{(4)}_q+qm^{(4)}_r),
\end{array}\]
where $0 \leq m^{(i)}_t \leq t-1$ for $t=p,q,r$ and $i=1,\ldots,4$ are uniquely determined by the $\theta_{s,t}$ for the respective values of $s$ and $t$. The bijection from (2.8) in \cite{FS}*{Remark~2.3} now maps $((m_p,p),(m_q,q),(m_r,r))$ to $((m^{(4)}_p,p),(m^{(4)}_q,q),(m^{(4)}_r,r))$. It is easy to check that $m^{(4)}_t = m_t$ for $t=p,q,r$, which shows that condition~(2.8) in \cite{FS}*{Remark~2.3} is valid. Applying \cite{KP}*{Definition 1.5} to the case of $\Lambda_{S,\theta}$, we see that $C^*(\Lambda_{S,\theta})$ is the universal $C^*$-algebra generated by isometries $(t_{(m,p)})_{p \in S, 0 \leq m \leq p-1}$ subject to the relations:
\[\begin{array}{c} \textnormal{(i)} \ t_{(m,p)}t_{(n,q)} = t_{(n',q)}t_{(m',p)} \text{ if } m+pn = n'+qm' \quad \text{and} \quad\textnormal{(ii)} \ \sum\limits_{m=0}^{p-1} t^{\phantom{*}}_{(m,p)}t_{(m,p)}^* = 1 \end{array}\]
for all $p,q \in S$.
\end{rem}

\begin{cor}\label{cor:tor subalgebra via k-graphs}
$\CA_S$ is isomorphic to $C^*(\Lambda_{S,\theta})$.
\end{cor}
\begin{proof}
We will work with $\QQ(U)$ in place of $\CA_S$ and invoke Proposition~\ref{prop:A_S as BQ of U}. Condition~(i) guarantees that $(m,p) \mapsto t_{(m,p)}$ yields a representation of $U$ by isometries as $U$ is generated by $(m,p)$ with $p \in S, 0 \leq m \leq p-1$, and $(m,p)(n,q) = (m+pn,pq) = (n',q)(m',p)$. (ii) holds for arbitrary $p \in H^+$ if we write $t_{(m,p)}$ for the product $t_{(m_1,p_1)}\cdots t_{(m_k,p_k)}$ where $(m_1,p_1)\cdots (m_k,p_k) = (m,p) \in U$ with $p_i \in S$. It is then straightforward to verify that we get a $*$-homomorphism $C^*(U) \to C^*(\Lambda_{S,\theta})$. 
Now let $F \subset U$ be a foundation set and set $h := \text{lcm}(\{h' : (m',h') \in F \text{ for some } 0 \leq m' \leq h'-1\})$. Then $F_a := \{ (m,h) : 0 \leq m \leq h-1\}$ is a foundation set that refines $F$. Therefore, it suffices to establish \eqref{eq:BQ} for $F_a$ in place of $F$. But as $F_a$ is accurate, \eqref{eq:BQ} takes the form $\sum_{m=0}^{h-1} t^{\phantom{*}}_{(m,h)}t_{(m,h)}^* = 1$, which follows from (ii) as explained in the proof of Proposition~\ref{prop:A_S as BQ of U}. Thus $v_{(m,p)} \mapsto t_{(m,p)}$ defines a surjective $*$-homomorphism $\QQ(U) \to C^*(\Lambda_{S,\theta})$. By simplicity, see Corollary~\ref{cor:UHF into BD seq split cr pr} and Proposition~\ref{prop:A_S as BQ of U}, this map is also injective.    
\end{proof}

\begin{rem}\label{rem:Lambda_S flip}
Similar to $\Lambda_{S,\theta}$, we can also consider the row-finite $k$-graph $\Lambda_{S,\sigma}$ with $\sigma_{p,q}$ being the flip, i.e.\ $\sigma_{p,q}(m,n) := (n,m)$. That is to say, we keep the skeleton of $\Lambda_{S,\theta}$, but replace $\theta$ by $\sigma$. In this case, it is easy to see that $C^*(\Lambda_{S,\sigma}) \cong \bigotimes_{p \in S} \CO_p$. 
\end{rem}

With regards to the $K$-theory of $\QQ_S$, it is interesting to ask whether $C^*(\Lambda_{S,\theta})$ and $C^*(\Lambda_{S,\sigma})$ are isomorphic or not. At least for $\lvert S \rvert \leq 2$, the answer is known to be positive.

\begin{prop} \label{prop:A_S for |S|=2}
Let $p,q \geq 2$ be two relatively prime numbers and $S = \{p,q\}$. Then $\CA_S \cong C^*(\Lambda_{S,\theta})\cong C^*(\Lambda_{S,\sigma}) \cong \CO_p \otimes \CO_q$. 
\end{prop}
\begin{proof}
We have seen in Corollary~\ref{cor:tor subalgebra via k-graphs} and Remark~\ref{rem:Lambda_S flip} that the UCT Kirchberg algebras $\CA_S$ and $\CO_p \otimes \CO_q$ are both expressible as $C^*$-algebras associated with row-finite $2$-graphs $\Lambda_{S,\theta}$ and $\Lambda_{S,\sigma}$ sharing the same skeleton. The claim therefore follows from \cite{Evans}*{Corollary~5.3}.
\end{proof}

Concerning a generalization of Proposition~\ref{prop:A_S for |S|=2} to the case of $\lvert S \rvert \geq 3$, we learned from Aidan Sims that the following conjecture for $k$-graphs might be true:

\begin{conj}\label{conj:k-graph}
Suppose $\Lambda$ and $\Lambda'$ are row-finite $k$-graphs without sources such that $C^*(\Lambda)$ and $C^*(\Lambda')$ are unital, purely infinite and simple. If $\Lambda$ and $\Lambda'$ have the same skeleton, then the associated $C^*$-algebras are isomorphic.
\end{conj}

Note that $C^*(\Lambda)$ and $C^*(\Lambda')$ are indeed unital UCT Kirchberg algebras, as separability, nuclearity and the UCT are automatically satisfied, see \cite{KP}*{Theorem~5.5}. We will come back to Conjecture~\ref{conj:k-graph} at the end of the next section.

\section{\texorpdfstring{Towards a classification of $\QQ_S$}{Towards a classification}}\label{sec:classification}

This final section provides a survey of the progress on the classification of $\QQ_S$ that we achieve through the preceding sections and a spectral sequence argument for $K_*(\CA_S)$, see Theorem~\ref{thm:main result} and Theorem~\ref{thm:K-theory for A_S}. Recall that $N = \Z\bigl[\{ \frac{1}{p} : p \in S\}\bigr]$ and $g_S$ denotes the greatest common divisor of $\{p-1 : p \in S\}$. We begin by stating our main result.

\begin{thm}\label{thm:main result}
Let $S \subset \N^\times\setminus\{1\}$ be a non-empty family of relatively prime numbers. Then the $K$-theory of $\QQ_S$ satisfies
\[\begin{array}{c}
K_{i}(\QQ_{S})\cong \Z^{2^{\lvert S \rvert-1}} \oplus K_i(\CA_S),\quad i=0,1,\end{array}\]
where $K_i(\CA_S)$ is a torsion group. Moreover, the following statements hold: 
\begin{enumerate}[(a)]
\item If $g_S=1$, then $K_i(\QQ_S)$ is free abelian in $2^{\lvert S \rvert-1}$ generators for $i=0,1$, and $[1]=0$.
\item If $\lvert S \rvert=1$, then $(K_{0}(\QQ_{S}),[1],K_{1}(\QQ_{S})) \cong (\Z \oplus \Z/g_S\Z, (0,1), \Z)$.
\item If $\lvert S \rvert=2$, then $(K_{0}(\QQ_{S}),[1],K_{1}(\QQ_{S})) \cong (\Z^2 \oplus \Z/g_S\Z, (0,1), \Z^2 \oplus \Z/g_S\Z)$.
\end{enumerate}
\end{thm}

\begin{rem}\label{rem:K-theory of Q_S mod tor}
Note that for $S = \left\lbrace p \right\rbrace$, the torsion subalgebra $\CA_S$ is canonically isomorphic to the Cuntz algebra $\CO_p$. Therefore, Theorem~\ref{thm:main result}~(b) recovers known results by Hirshberg \cite{Hir}*{Example~1, p.~106} and Katsura \cite{KatsuraIV}*{Example~A.6}. Indeed, it is already clear from the presentation for $\CO(E_{p,1})$ described in \cite{KatsuraIV}*{Example~A.6} that it coincides with $\QQ_S$. Theorem~\ref{thm:main result}~(c) shows an unexpected result for the $K$-groups of $\QQ_S$ in the case of $S=\{p,q\}$ for two relatively prime numbers $p$ and $q$ with $g_S > 1$: $K_1(\QQ_S)$ has torsion and is therefore, for instance, not a graph $C^*$-algebra, see \cite{RS}*{Theorem~3.2}. By virtue of (a), Theorem~\ref{thm:main result} also explains why $\QQ_\N$ and $\QQ_{2}$ have torsion free $K$-groups. More importantly, it shows that the presence of $2$ in the family $S$ is not the only way to achieve this. Indeed, $S$ can contain at most one even number. If $g_S=1$, then $S$ must contain an even number, and there are many examples, e.g.\ $S$ with $2^m+1,2n \in S$ for some $m,n \geq 1$.
\end{rem}

In view of the Kirchberg-Phillips classification theorem \cites{Kir,Phi}, we get the following immediate consequence of Theorem~\ref{thm:main result}.

\begin{cor}\label{cor:isomorphism classes for |S| at most 2}
Let $S,T \subset \N^\times \setminus\{1\}$ be non-empty families of relatively prime numbers. Then $\QQ_S \cong \QQ_T$ implies $\lvert S \rvert = \lvert T \rvert$. Moreover, the following statements hold:
\begin{enumerate}[(a)]
\item If $g_S=1=g_T$, then $\QQ_S$ is isomorphic to $\QQ_T$ if and only if $\lvert S \rvert=\lvert T \rvert$. 
\item If $\lvert S \rvert \leq 2$, then $\QQ_S$ is isomorphic to $\QQ_T$ if and only if $\lvert S \rvert=\lvert T \rvert$ and $g_S=g_T$.
\end{enumerate}
\end{cor}

Observe that the decomposition of $K_*(\QQ_S)$ claimed in Theorem~\ref{thm:main result} follows from Corollary~\ref{cor:torsion and free part K-theory}, Proposition~\ref{prop:K-theory torsion free part} and Corollary~\ref{cor:subalgebra for torsion part}. To prove our main result, it is therefore enough to establish the following theorem reflecting our present knowledge on the torsion subalgebra $\CA_S$, an object which is certainly of interest in its own right.

\begin{thm}\label{thm:K-theory for A_S}
Let $S \subset \N^\times\setminus\{1\}$ be a non-empty family of relatively prime numbers. Then the following statements hold:
\begin{enumerate}[(a)]
\item If $g_S=1$, then $\CA_S \cong \CO_2 \cong \bigotimes_{p \in S} \CO_p$.
\item If $S = \left\lbrace p \right\rbrace$, then $\CA_S \cong \CO_p$.
\item If $S = \left \lbrace p,q \right\rbrace$ with $p\neq q$, then $\CA_S \cong \CO_p \otimes \CO_q$.
\item For $\lvert S \rvert\geq 3$ and $g_S > 1$, $K_i(\CA_S)$ is a torsion group in which the order of any element divides $g_S^{2^{\lvert S \rvert-2}}$. Moreover, $K_i(\CA_S)$ is finite whenever $S$ is finite.
\end{enumerate}
\end{thm}

Note that in the case of infinite $S$ with $g_S > 1$, part (d) still makes sense within the realm of supernatural numbers. Based on Theorem~\ref{thm:main result} and Theorem~\ref{thm:K-theory for A_S}, we suspect that the general situation is in accordance with Conjecture~\ref{conj:k-graph}: 

\begin{conj}\label{conj:K-theory of QQ_S}
For a family $S \subset \N^\times\setminus\{1\}$ of relatively prime numbers with $\lvert S\rvert\geq 2$, $\CA_S$ is isomorphic to $\bigotimes_{p \in S} \CO_p$. Equivalently, $\QQ_S$ is the unital UCT Kirchberg algebra with 
\[
(K_0(\QQ_S),[1],K_1(\QQ_S)) = (\Z^{2^{\lvert S \rvert-1}} \oplus (\Z/g_S\Z)^{2^{\lvert S \rvert-2}},(0,e_1),\Z^{2^{\lvert S \rvert-1}} \oplus (\Z/g_S\Z)^{2^{\lvert S \rvert-2}}),
\]
where $e_1 = (\delta_{1,j})_j \in (\Z/g_S\Z)^{2^{\lvert S \rvert-2}}$. In particular, if $S,T \subset \N^\times \setminus\{1\}$ are non-empty sets of relatively prime numbers, then $\QQ_S$ is isomorphic to $\QQ_T$ if and only if $\lvert S \rvert=\lvert T \rvert$ and $g_S=g_T$.
\end{conj}

\begin{rem}\label{rem:stable relations}
It follows from Theorem~\ref{thm:main result} and Theorem~\ref{thm:K-theory for A_S}~(d) that the $K$-theory of $\QQ_S$ is finitely generated if and only if $S$ is finite. Consequently, when $S$ is finite the defining relations of $\QQ_S$ from Defintion~\ref{def:Q_S} are \emph{stable}, see \cite{enders1}*{Corollary~4.6} and \cite{loring1}*{Chapter~14}.
\end{rem}

For the proof of Theorem~\ref{thm:K-theory for A_S}, we will employ the isomorphism $\CA_S \cong M_{d^\infty} \rtimes_\alpha^e H$ and make use of a spectral sequence by Kasparov constructed in \cite{kasparov}*{6.10}. Let us briefly review the relevant ideas and refer to \cite{barlak15} for a detailed exposition. Given a $C^*$-dynamical system $(B,\beta,\Z^k)$, we can consider its \emph{mapping torus}
\[
\M_\beta(B):=\left\lbrace f\in C(\R^k,B)\ : \ \beta_z (f(x))=f(x+z)\ \text{for all}\ x\in \R^k,\ z\in \Z^k \right\rbrace.
\]
It is well-known that $K_*(\M_\beta(B))$ is isomorphic to $K_{*+k}(B\rtimes_\beta\Z^k)$, see e.g.\ \cite{barlak15}*{Section~1}. The mapping torus admits a finite cofiltration 
\begin{equation}
\label{cofiltrationMappingTorus}
\M_\beta(B)=F_k \stackrel{\pi_k}{\onto} F_{k-1} \stackrel{\pi_{k-1}}{\onto} \cdots \stackrel{\pi_1}\onto F_0=A \stackrel{\pi_0}{\onto} F_{-1}=0
\end{equation}
arising from the filtration of $\R^k$ by its skeletons
\[
\emptyset = X_{-1} \subset \Z^k = X_0 \subset X_1 \subset \cdots\subset X_k = \R^k,
\]
where $X_\ell := \{ (x_1,\ldots,x_k) \in \R^k : \lvert\{ 1 \leq i \leq k : x_i \in \R\setminus\Z\}\rvert \leq \ell\}$. 

As for filtrations of $C^*$-algebras by closed ideals \cite{schochet}, there is a standard way relying on Massey's technique of exact couples \cites{massey1,massey2} of associating a spectral sequence to a given finite cofiltration of a $C^*$-algebra. In this way, the cofiltration \eqref{cofiltrationMappingTorus} yields a spectral sequence $(E_\ell,d_\ell)_{\ell\geq 1}$ that converges to $K_*(\M_\beta(B))\cong K_{*+k}(B\rtimes_\beta \Z^k)$. Using Savinien-Bellissard's \cite{savinienBellissard} description of the $E_1$-term, we can summarize as follows.

\begin{thm}[cf.\ {\cite{kasparov}*{6.10}}, {\cite{savinienBellissard}*{Theorem~2}} and {\cite{barlak15}*{Corollary~2.5}}]\label{thm:spec seq for cr prod}~\newline
Let $(B,\beta,\Z^k)$ be a $C^*$-dynamical system. There exists a cohomological spectral sequence $(E_\ell,d_\ell)_{\ell\geq 1}$ converging to $K_*(\M_\beta(B))\cong K_{*+k}(B\rtimes_\beta \Z^k)$. The $E_1$-term is given by
\[
\begin{array}{l}
E_1^{p,q}:= K_q(B) \otimes_\Z \Lambda^p(\Z^k),\text{ with}\\
d_1^{p,q}\colon E_1^{p,q}\to E_1^{p+1,q}, \quad x\otimes e\mapsto \sum\limits_{j=1}^k (K_q(\beta_j)-\op{id})(x)\otimes (e_j\wedge e).
\end{array}
\]
Furthermore, the spectral sequence collapses at the $(k+1)$th page, so that $E_\infty = E_{k+1}$.
\end{thm}

By Bott periodicity, we have that $(E_\ell^{p,q+2},d_\ell^{p,q+2})=(E_\ell^{p,q},d_\ell^{p,q})$ for all $p,q \in \Z$. In particular, the $E_\infty$-term reduces to $E_\infty^{p,q}$ with $p\in \Z$ and $q=0,1$.

\begin{rem}
Let us recall the meaning of convergence of the spectral sequence $(E_\ell,d_\ell)_{\ell\geq 1}$. For $q=0,1$, consider the diagram
\[
K_q(\M_\gamma(B))=K_q(F_k) \longrightarrow K_q(F_{k-1}) \longrightarrow \cdots \longrightarrow K_q(F_0) \longrightarrow K_q(F_{-1})=0.
\]
Define $\F_p K_q(\M_\beta(B)):=\op{ker}(K_q(\M_\beta(B))\to K_q(F_p))$ for $p=-1,\ldots,k$, and observe that this gives rise to a filtration of abelian groups
\[
0 \into \F_{k-1} K_q(\M_\beta(B)) \into \cdots \into	\F_{-1} K_q(\M_\beta(B)) = K_q(\M_\beta(B)).
\]
One can now show the existence of exact sequences
\begin{equation}\label{eq:exact sequences from filtration}
0 \longrightarrow \F_p K_{p+q}(\M_\beta(B)) \longrightarrow \F_{p-1} K_{p+q}(\M_\beta(B)) \longrightarrow E_{\infty}^{p,q} \longrightarrow 0,
\end{equation}
or in other words, there are isomorphisms
\[
E_{\infty}^{p,q}\cong \F_{p-1} K_{p+q}(\M_\beta(B))/\F_p K_{p+q}(\M_\beta(B)).
\]
Hence, the $E_\infty$-term determines the $K$-theory of $\M_\beta(B)$, and thus of $B\rtimes_\beta \Z^k$, up to group extension problems.
\end{rem}

Let us now turn to the $K$-theory of $M_{d^\infty} \rtimes_\alpha^e H^+$. By Laca's dilation theorem \cite{Lac}, see also Remark~\ref{rem:dilations}, we may and will determine the $K$-theory of the dilated crossed product $M_{d^\infty,\infty} \rtimes_{\alpha_\infty} H$ instead. Fix a natural number $1 \leq k \leq \lvert S \rvert$ and observe that $H_k \cong \Z^k$.  Let $\alpha_\infty(k)$ be the $H_k$-action on $M_{d^\infty,\infty}$ induced by the $k$ smallest elements $p_1 < p_2 < \dotsb < p_k$ of $S$. It follows from the proof of Proposition~\ref{prop:UHF cross prod torsion K-th} that $K_0(\alpha_{\infty,p_\ell})$ is given by multiplication with $1/p_\ell$ on $K_0(M_{d^\infty}) \cong N$. It turns out to be more convenient to work with the action $\alpha_\infty^{-1}(k)$ given by the inverses of the $\alpha_\ell$, whose crossed product is canonically isomorphic to $M_{d^\infty,\infty} \rtimes_{\alpha_\infty(k)} H_k$.

Let $(E_\ell,d_\ell)_{\ell\geq 1}$ denote the spectral sequence associated with $\alpha^{-1}_\infty(k)$. As $K_1(M_{d^\infty})=0$, it follows directly from Theorem~\ref{thm:spec seq for cr prod} that $E_1^{p,1}=0$ for all $p \in \Z$. Moreover, according to Theorem~\ref{thm:spec seq for cr prod}, $d_1^{p,0}\colon N\otimes_\Z \Lambda^p(\Z^k)\to N\otimes_\Z \Lambda^p(\Z^k)$, $p \in \Z$, is given by
\[\begin{array}{c}
d_1^{p,0}(x\otimes e)=\sum_{\ell=1}^k (p_\ell -1)x\otimes e_\ell\wedge e=\sum_{\ell=1}^k x\otimes (p_\ell -1)e_\ell\wedge e. 
\end{array}\]
In other words, $d_1^{p,0} = \id_N\otimes h^p$ with
\begin{equation}\label{eq:differential formula}
\begin{array}{c}
h^p\colon \Lambda^p(\Z^k)\to \Lambda^{p+1}(\Z^k),\quad h^p(e)=\sum_{\ell=1}^k (p_\ell-1)e_\ell\wedge e.
\end{array}
\end{equation}
To obtain $E_2^{p,0}$, we therefore compute the cohomology of the complex $(\Lambda^p(\Z^k),h^p)_{p\in \Z}$. To do so, we consider $h^p$ as a matrix $A_p\in M_{\binom k {p+1} \times \binom k p}(\Z)$, where the identification is taken with respect to the canonical bases of $\Lambda^p(\Z^k)$ and $\Lambda^{p+1}(\Z^k)$ in lexicographical ordering. The computation then mainly reduces to determining the Smith normal form of $A_p$.

\begin{thm}[Smith normal form]\label{thm:Smith normal form}
Let $A$ be a non-zero $m\times n$-matrix over a principal ideal domain $R$. There is an invertible $m\times m$-matrix $S$ and an invertible $n\times n$-matrix $T$ over $R$, so that
\[D:=SAT=\op{diag}(\delta_1,\ldots,\delta_r,0,\ldots,0)\]
for some $r \leq \min(m,n)$ and non-zero $\delta_i\in R$ satisfying $\delta_i|\delta_{i+1}$ for $1 \leq i \leq r-1$. The elements $\delta_i$ are unique up to multiplication with some unit and are called \emph{elementary divisors} of $A$. The diagonal matrix $D$ is called a \emph{Smith normal form} of $A$. The $\delta_i$ can be computed as
\begin{equation}\label{eq:det divisor formula}
\begin{array}{c}
\delta_1 = d_1(A),\quad \delta_i = \frac{d_i(A)}{d_{i-1}(A)},
\end{array}
\end{equation}
where $d_i(A)$, called the \emph{$i$-th determinant divisor}, is the greatest common divisor of all $i\times i$-minors of $A$.
\end{thm}

Of course, $D$ can only be a diagonal matrix if $m=n$. The notation in Theorem~\ref{thm:Smith normal form} is supposed to mean that $D$ is the $m\times n$ matrix over $R$ with the $\min(m,n)\times \min(m,n)$ left upper block matrix being $\op{diag}(\delta_1,\ldots,\delta_r,0,\ldots,0)$ and all other entries being zero.

For each $1\leq k\leq\lvert S\rvert$, set $g_k :=\gcd(\{p_\ell-1 : \ell=1,\cdots,k\})$.
\begin{lem}\label{lem:ker d^p / im d^(p-1)}
The group $\op{ker}(h^p)/\op{im}(h^{p-1})$ is isomorphic to $(\Z/g_k\Z)^{\binom{k-1}{p-1}}$ for $1 \leq p \leq k$ and vanishes otherwise. 
\end{lem}
\begin{proof}
For $p\in \Z$, let $D_p=S_pA_pT_p$ denote the Smith normal form of $A_p$ with elementary divisors $\delta^{(p)}_1,\ldots,\delta^{(p)}_{r_p}$. As $\Lambda^p(\Z^k) = 0$ unless $0 \leq p\leq k$, $\op{ker}(h^p)/\op{im}(h^{p-1})$ vanishes if $p < 0$ or $p \geq k+1$.

If $p=0$, then $h^0\colon \Z \to \Z^k$ is given by $A_0 = (p_1-1,\ldots,p_k-1)$. Thus we have $r_0=1$ and $\delta^{(0)}_1=g_k$. Moreover, $h^0$ is injective, so $\op{ker}(h^0)/\op{im}(h^{-1})=0$.

Likewise, $p=k$ is simple as $h^{k-1}\colon \Z^k \to \Z$ is given by $A_{k-1} = (p_1-1,\ldots,p_k-1)^t$ and $h^k$ is zero because $\Lambda^{k+1}(\Z^k) = 0$. Therefore, $r_{k-1}=1$ and $\delta^{(k-1)}_1=g_k$, and hence 
\[
\op{ker}(h^k)/\op{im}(h^{k-1}) = \Z/g_k\Z=(\Z/g_k\Z)^{\binom{k-1}{k-1}}.
\]
As this completes the proof for $p\leq0$ and $p\geq k$, we will assume $1\leq p\leq k-1$ from now on. 

We start by showing that for $\ell=1,\cdots,k$, the matrix $A_p$ contains a $\binom{k-1}{p}\times \binom{k-1}p$ diagonal matrix with entries $\pm (p_\ell-1)$ (obtained by deleting suitable rows and columns). This will allow us to conclude that $A_p$ has a $j \times j$-minor equal to $\pm(p_\ell-1)^j$ for each $j=1,\ldots,\binom{k-1}p$. Thus we obtain that $d_j(A_p)$ divides $g_k^j$ for $j=1,\ldots,\binom{k-1}{p}$: 

First, keep only those columns of $A_p$ which correspond to basis elements $e_{i_1}\wedge\ldots\wedge e_{i_p} \in \Lambda^p(\Z^k)$ satisfying $\ell \neq i_j$ for all $j=1,\ldots,p$. As this amounts to choosing $p$ elements out of $k-1$ without order and repetition, we are left with $\binom{k-1}p$ columns (out of $\binom k p$). Next, we restrict to those rows which correspond to basis elements $e_{i_1}\wedge\ldots\wedge e_{i_{p+1}} \in \Lambda^{p+1}(\Z^k)$ satisfying $\ell=i_j$ for some (necessarily unique) $j=1,\ldots,p-1$. Here again $\binom{k-1}p$ rows (out of $\binom k {p+1}$) remain. The resulting matrix describes the linear map 
\[
\Lambda^p(\Z^k) \supset \Z^{\binom{k-1}p}\to \Z^{\binom{k-1}p} \subset \Lambda^{p+1}(\Z^k),\ e_{i_1}\wedge\ldots\wedge e_{i_p}\mapsto (p_\ell-1)\cdot e_\ell\wedge e_{i_1}\wedge\ldots\wedge e_{i_p},
\]
which is nothing but a diagonal matrix of size $\binom{k-1}p$ with entries $\pm(p_\ell-1)$. As explained above, we thus obtain that $d_j(A_p)|g_k^j$ for $j=1,\ldots,\binom{k-1}{p}$.

We will now show that the converse holds as well, i.e.\ $g_k^j|d_j(A_p)$ for $j=1,\ldots,\binom{k-1}{p}$. Note that every $1\times 1$-minor is either zero or $p_\ell-1$ for some $\ell=1,\ldots,k$. This shows that $d_1(A_p)=g_k$ for $1 \leq p \leq k-1$. Let $1\leq j\leq \binom{k-1}{p}-1$ and assume that $d_j(A_p)=g_k^j$. Let $L$ be any $(j+1)\times (j+1)$-matrix arising from $A_p$ by deleting sums and rows. By the Laplace expansion theorem, the determinant of $L$ is given as a linear combination of some of its $j\times j$-minors. The coefficients in the linear combination all are entries of $L$. The occurring minors are all $j\times j$-minors of $A_p$. Hence, $g_k^j|\op{det}(L)$ by assumption. In fact, we have $g_k^{j+1}|\op{det}(L)$ because all entries in $A_p$ are divisible by $g_k$. Altogether, $d_j(A_p)=g_k^j$ for $j=1,\ldots,\binom{k-1}{p}$ and we have shown that for $p=1,\ldots,k-1$, $r_p\geq \binom{k-1}{p}$ and $\delta^{(p)}_j=g_k$ for $j=1,\ldots,\binom{k-1}p$.

Since $A_p$ and $D_p$ have isomorphic kernel and image, our considerations show that
\[\begin{array}{lcr}
\binom{k-1}p \leq \op{rank}(\op{im}(h^p)) &\quad\text{ and }\quad& 
\op{rank}(\op{ker}(h^p)) \leq \binom k p -\binom{k-1}p=\binom{k-1}{p-1}.
\end{array}\]
By $h^{p+1}\circ h^p=0$, we conclude that $\op{rank}(\op{ker}(h^{p+1})) = \op{rank}(\op{im}(h^p)) = \binom{k-1}p$ which implies $r_p=\binom{k-1}p$. Moreover, $h^{p}\circ h^{p-1}=0$ forces $T_p^{-1}S_{p-1}^{-1}(\op{im}(D_{p-1}))\subset\op{ker}(D_p)$ or, equivalently, $\op{im}(D_{p-1})\subset\op{ker}(D_pT_p^{-1}S_{p-1}^{-1})$. Since 
\[\op{im}(D_{p-1})=g_k\Z^{\binom{k-1} {p-1}}\oplus\{0\}^{\binom k p - \binom{k-1} {p-1}}\]
has the same rank as $\op{ker}(D_pT_p^{-1}S_{p-1}^{-1})$,
it means that
\[\op{ker}(D_pT_p^{-1}S_{p-1}^{-1})=\Z^{\binom{k-1}{p-1}}\oplus\{0\}^{\binom k p -\binom{k-1}{p-1}}.\]
Moreover, $S_{p-1}$ is an automorphism of $\Z^{\binom k p}$ that restricts both to an isomorphism $\op{ker}(A_p)  \stackrel{\cong}{\longrightarrow} \op{ker}(D_pT_p^{-1}S_{p-1}^{-1})$ and to an isomorphism $\op{im}(A_{p-1}) \stackrel{\cong}{\longrightarrow} \op{im}(D_{p-1})$.
Hence,
\[\begin{array}{lclcl}
\op{ker}(h^p) / \op{im}(h^{p-1}) &=& \op{ker}(A_p) / \op{im}(A_{p-1}) & \cong & \op{ker}(D_p T_p^{-1}S_{p-1}^{-1}) / \op{im}(D_{p-1}) \vspace*{2mm}\\ 
&&&\cong& (\Z/g_k\Z)^{\binom{k-1}{p-1}}.
\end{array}\]
\end{proof}

Lemma~\ref{lem:ker d^p / im d^(p-1)} now allows us to compute the $E_2$-term of the spectral sequence associated to $\alpha_{\infty}^{-1}(k)\colon H_k\curvearrowright M_{d^\infty,\infty}$ by appealing to the following simple, but useful observation.

\begin{lem}\label{lem:N/gN}
The group $N/g_SN$ is isomorphic to $\Z/g_S\Z$. Moreover, for every $1\leq k\leq \lvert S\rvert$, the group $N/g_kN$ is isomorphic to a subgroup of $\Z/g_k\Z$.
\end{lem}
\begin{proof}
Recall that $S$ consists of relatively prime numbers, $N=\Z\bigl[\{\frac{1}{p}:p\in S\}\bigr]$ and let us simply write $g$ for $g_S=\gcd(\{p-1:p\in S\})$. The map
\[\begin{array}{c} N/gN \to \Z/g\Z, \quad \frac{1}{r}+gN \mapsto s + g\Z, \end{array}\]
where $r$ is a natural number and $s$ is the unique solution in $\{0,1,\dotsc,g-1\}$ of $rs=1 \pmod{g}$, defines a group homomorphism. To see this, note first that for every $p\in \PS$ there is a $q\in S$ such that $p|q$, i.e.\ $\gcd{(q-1,p)}=1$. Therefore, $\gcd{(g,p)}=1$ for all $p\in \PS$. If $\frac{1}{r}\in N$, then all the prime factors of $r$ come from $\PS$, and it follows that $\gcd{(g,r)}=1$. Thus, the above map is well-defined and extends by addition to the whole domain. Moreover, every $s$ appearing as a solution is relatively prime with $g$, meaning that the kernel is $gN$, i.e.\ the map is injective. Finally, the inverse map is given by $1+g\Z\mapsto 1+gN$.

For the second part, set $g_k'=g_k/\max{(\gcd{(g_k,r)})}$, where the maximum is taken over all natural numbers $r$ such that $\frac{1}{r}\in N$, i.e.\ $g_k'$ is the largest number dividing $g_k$ so that $\gcd{(g_k',r)}=1$ for all such $r$. Then $g_k'N=g_kN$ and a similar proof as above shows that $N/g_k N=N/g_k'N\cong\Z/g_k'\Z$.
\end{proof}

\begin{prop}\label{prop:E for alpha(k)} 
For every $1 \leq k \leq \lvert S \rvert$, the respective group $E_2^{p,0}$ is isomorphic to a subgroup of $(\Z/g_k\Z)^{\binom{k-1}{p-1}}$ for $1 \leq p \leq k$, and vanishes otherwise. $E_2^{p,1}$ vanishes for $p \in \Z$.
\end{prop}
\begin{proof}
Note that $N$ is torsion free and hence a flat module over $\Z$. Thus, an application of Lemma~\ref{lem:ker d^p / im d^(p-1)} yields
\[\begin{array}{lclclcl}
E_2^{p,0} &\hspace*{-1mm}=\hspace*{-1mm}& \op{ker}(\id_{N}\otimes h^p)/\op{im}(\id_{N}\otimes h^{p-1}) &\hspace*{-1mm}\cong\hspace*{-1mm}& N\otimes_\Z \op{ker}(h^p)/\op{im}(h^{p-1}) \vspace*{2mm}\\
&\hspace*{-1mm}\cong\hspace*{-1mm}& N \otimes_\Z (\Z/g_k\Z)^{\binom {k-1}{p-1}} &\hspace*{-1mm}\cong\hspace*{-1mm}& (N \otimes_\Z \Z/g_k\Z)^{\binom {k-1}{p-1}} &\hspace*{-1mm}\cong\hspace*{-1mm}& (N/g_kN)^{\binom {k-1}{p-1}}
\end{array}\]
and Lemma~\ref{lem:N/gN} shows that $N/g_kN$ is isomorphic to a subgroup of $\Z/g_k\Z$. The second claim follows from the input data.
\end{proof}

\begin{rem}\label{rem:apps for E2term lemma}
Assume that $g_k=1$ for some $1 \leq k \leq \lvert S \rvert$, $k < \infty$, and let $k \leq \ell \leq \lvert S \rvert$ be a natural number. If $(E_i^{p,q})_{i \geq 1}$ denotes the spectral sequence associated with $\alpha_\infty^{-1}(\ell)$, then Proposition~\ref{prop:E for alpha(k)} yields $E_2^{p,0}=0$  for all $p \in \Z$.
\end{rem}

\begin{proof}[Proof of Theorem~\ref{thm:K-theory for A_S}]
Let $k \geq 1$ be finite with $k \leq \lvert S \rvert$. The main idea is to use the $E_\infty$-term of the spectral sequence associated with $\alpha^{-1}_\infty(k)$ to compute $K_*(M_{d^\infty} \rtimes^e_{\alpha(k)} H^+_k)$, up to certain group by employing convergence of this spectral sequence, see Theorem~\ref{thm:spec seq for cr prod}. Recall the general form \eqref{eq:exact sequences from filtration} of the extension problems involved. Since $\F_k K_q(\M_{\alpha_\infty}(M_{d^\infty,\infty})) =0$ and hence $\F_{k-1} K_q(\M_{\alpha_\infty}(M_{d^\infty,\infty})) \cong  E_{\infty}^{k,q-k}$, we face $k$ iterative extensions of the form: 
\begin{equation}\label{eq:iterative extensions}
\begin{array}{ccccc}
E_{\infty}^{k,q-k} &\hspace*{-2mm}\into\hspace*{-2mm}& \F_{k-2} K_q(\M_{\alpha_\infty}(M_{d^\infty,\infty})) &\hspace*{-2mm}\onto\hspace*{-2mm}& E_{\infty}^{k-1,q-k+1} \vspace*{2mm}\\
\F_{k-2} K_q(\M_{\alpha_\infty}(M_{d^\infty,\infty})) &\hspace*{-2mm}\into\hspace*{-2mm}& \F_{k-3} K_q(\M_{\alpha_\infty}(M_{d^\infty,\infty})) &\hspace*{-2mm}\onto\hspace*{-2mm}& E_{\infty}^{k-2,q-k+2} \vspace*{0mm}\\
\vdots&&\vdots&&\vdots \vspace*{0mm}\\
\F_{0} K_{q}(\M_{\alpha_\infty}(M_{d^\infty,\infty})) &\hspace*{-2mm}\into\hspace*{-2mm}& \F_{-1} K_{q}(\M_{\alpha_\infty}(M_{d^\infty,\infty})) &\hspace*{-2mm}\onto\hspace*{-2mm}& E_{\infty}^{0,q}.
\end{array}
\end{equation}
Using
\[\begin{array}{lcl}
\F_{-1} K_{q}(\M_{\alpha_\infty}(M_{d^\infty,\infty})) &=& K_{q}(\M_{\alpha_\infty}(M_{d^\infty,\infty})) \\ 
&\cong& K_{k+q}(M_{d^\infty,\infty} \rtimes_{\alpha_\infty(k)} H_k) \cong K_{k+q}(M_{d^\infty} \rtimes^e_{\alpha(k)} H^+_k),
\end{array}\]
we will thus arrive at $K_{k+q}(M_{d^\infty} \rtimes^e_{\alpha(k)} H^+_k)$, see Theorem~\ref{thm:spec seq for cr prod}. Recall that by Bott periodicity, $E_\infty^{p,q+2} \cong E_\infty^{p,q}$ for all $q \in \Z$. In addition, we know from Proposition~\ref{prop:E for alpha(k)} that for $p \in \Z$, the group $E_\infty^{p,1}$ is trivial, and $E_\infty^{p,0}$ vanishes unless $1 \leq p \leq k$, in which case it is a subquotient of $E_2^{p,0}$, and thus a subquotient of $(\Z/g_k\Z)^{\binom {k-1}{p-1}}$.

Assume now that $g=1$. Clearly, this holds exactly if $g_k = 1$ for some $k \geq 1$. For such $k \geq 1$, the corresponding $E_\infty$-term is trivial, yielding $K_*(M_{d^\infty} \rtimes^e_{\alpha(k)} H^+_k) = 0$. Using continuity of $K$-theory if necessary, we obtain that $\CA_S \cong M_{d^\infty} \rtimes^e_{\alpha} H^+$ has trivial $K$-theory. It follows from Kirchberg-Phillips classification that $\CA_S \cong \CO_2$. This proves (a).

If $S = \{p\}$, $\CA_S \cong \CO_p$ by the definition of $\CA_S$, and (b) follows. 

Claim (c) is nothing but Proposition~\ref{prop:A_S for |S|=2}.

Lastly, let us prove claim (d). Let $k \geq 2$, and denote by $(E_\ell,d_\ell)_{\ell \geq 1}$ the spectral sequence associated with $\alpha^{-1}_\infty(k)$. Recall that only those $E_\infty^{\ell,q-\ell}$ with $1 \leq \ell \leq k$ and $q-\ell \in 2\Z$ may be non-trivial subgroups of $(\Z/g_k\Z)^{\binom{k-1}{\ell-1}}$. Keeping track of the indices, we get
\[\begin{array}{c}
\sum\limits_{\substack{1 \leq \ell \leq k:\\ \ell \text{ even}}} \binom {k-1} {\ell-1} =2^{k-2}  = \sum\limits_{\substack{1 \leq \ell \leq k:\\ \ell \text{ odd}}} \binom {k-1} {\ell-1}.
\end{array}\]
This allows us to conclude that every element in $K_i(M_{d^\infty} \rtimes^e_{\alpha(k)} H^+_k)$ is a divisor of $g_k^{2^{k-2}}$. This concludes the proof, as $g_S = g_k$ for $k \leq \lvert S\rvert$ sufficiently large.
\end{proof}

\begin{rem}\label{rem:K-theory for M_d^infty}
It is possible to say something about the case $\lvert S \rvert=3$, though the answer is incomplete: Noting that $E_{\infty}^{2,-2} $ is a subgroup of $(\Z/g\Z)^2$, $E_{\infty}^{3,-2}$ and  $E_{\infty}^{1,0}$ are subgroups of $\Z/g\Z$, and the remaining terms vanish, we know that $K_1(M_{d^\infty} \rtimes^e_\alpha H^+) \cong E_\infty^{2,-2}$ and $K_0(M_{d^\infty} \rtimes^e_\alpha H^+)$ fits into an exact sequence
\[E_\infty^{3,-2} \into K_0(M_{d^\infty} \rtimes^e_\alpha H^+) \onto E_\infty^{1,0}.\]
But we cannot say more without additional information here.
\end{rem}

\begin{rem}
By considering $\CA_S$ as the $k$-graph $C^*$-algebra $C^*(\Lambda_{S,\theta})$ for finite $S$, see Corollary~\ref{cor:tor subalgebra via k-graphs}, one could probably also apply Evans' spectral sequence \cite{Evans}*{Theorem~3.15} to obtain Theorem~\ref{thm:K-theory for A_S} by performing basically the same proof. In fact, Evans' spectral sequence is the homological counterpart of the spectral sequence used here.
\end{rem}


\section*{References}
\begin{biblist}

\bib{barlak15}{article}{
  author={Barlak, Sel\c {c}uk},
  title={On the spectral sequence associated with the Baum-Connes Conjecture for $\mathbb Z^n$},
  note={\href {http://arxiv.org/abs/1504.03298}{arxiv:1504.03298v2}},
}

\bib{BarSza1}{article}{
  label={BaSz},
  author={Barlak, Sel\c {c}uk},
  author={Szab\'{o}, G\'{a}bor},
  title={Sequentially split $*$-homomorphisms between $C^*$-algebras},
  note={\href {http://arxiv.org/abs/1510.04555}{arxiv:1510.04555v2}},
}

\bib{BD}{article}{
  author={Bunce, John W.},
  author={Deddens, James A.},
  title={A family of simple $C^*$-algebras related to weighted shift operators},
  journal={J. Funct. Anal.},
  volume={19},
  year={1975},
  pages={13--24},
}

\bib{BaHLR}{article}{
  author={Brownlowe, Nathan},
  author={an Huef, Astrid},
  author={Laca, Marcelo},
  author={Raeburn, Iain},
  title={Boundary quotients of the {T}oeplitz algebra of the affine semigroup over the natural numbers},
  journal={Ergodic Theory Dynam. Systems},
  volume={32},
  year={2012},
  number={1},
  pages={35--62},
  issn={0143-3857},
  doi={\href {http://dx.doi.org/10.1017/S0143385710000830}{10.1017/S0143385710000830}},
}

\bib{BLS1}{article}{
  author={Brownlowe, Nathan},
  author={Larsen, Nadia S.},
  author={Stammeier, Nicolai},
  title={On $C^*$-algebras associated to right LCM semigroups},
  journal={Trans. Amer. Math. Soc.},
  year={2016},
  doi={\href {http://dx.doi.org/10.1090/tran/6638}{10.1090/tran/6638}},
}

\bib{BLS2}{article}{
  author={Brownlowe, Nathan},
  author={Larsen, Nadia S.},
  author={Stammeier, Nicolai},
  title={$C^*$-algebras of algebraic dynamical systems and right LCM semigroups},
  note={\href {http://arxiv.org/abs/1503.01599}{arxiv:1503.01599v1}},
}

\bib{BRRW}{article}{
  author={Brownlowe, Nathan},
  author={Ramagge, Jacqui},
  author={Robertson, David},
  author={Whittaker, Michael F.},
  title={Zappa-{S}z\'{e}p products of semigroups and their $C^*$-algebras},
  journal={J. Funct. Anal.},
  volume={266},
  year={2014},
  number={6},
  pages={3937--3967},
  issn={0022-1236},
  doi={\href {http://dx.doi.org/10.1016/j.jfa.2013.12.025}{10.1016/j.jfa.2013.12.025}},
}

\bib{bsBQforADS}{article}{
  label={BrSt},
  author={Brownlowe, Nathan},
  author={Stammeier, Nicolai},
  title={The boundary quotient for algebraic dynamical systems},
  journal={J. Math. Anal. Appl.},
  volume={438},
  year={2016},
  number={2},
  pages={772--789},
  doi={\href {http://dx.doi.org/10.1016/j.jmaa.2016.02.015}{10.1016/j.jmaa.2016.02.015}},
}

\bib{Com}{article}{
  author={Combes, Fran\c {c}ois},
  title={Crossed products and {M}orita equivalence},
  journal={Proc. Lond. Math. Soc. (3)},
  volume={49},
  year={1984},
  number={2},
  pages={289--306},
  issn={0024-6115},
  doi={\href {http://dx.doi.org/10.1112/plms/s3-49.2.289}{10.1112/plms/s3-49.2.289}},
}

\bib{CrispLaca}{article}{
  author={Crisp, John},
  author={Laca, Marcelo},
  title={Boundary quotients and ideals of {T}oeplitz {$C^*$}-algebras of {A}rtin groups},
  journal={J. Funct. Anal.},
  volume={242},
  year={2007},
  number={1},
  pages={127--156},
  issn={0022-1236},
  doi={\href {http://dx.doi.org/10.1016/j.jfa.2006.08.001}{10.1016/j.jfa.2006.08.001}},
}

\bib{CunPV}{article}{
  author={Cuntz, Joachim},
  title={A class of $C^*$-algebras and topological Markov chains II. Reducible chains and the Ext-functor for $C^*$-algebras},
  journal={Invent. Math.},
  year={1981},
  pages={25--40},
  volume={63},
  doi={\href {http://dx.doi.org/10.1007/BF01389192}{10.1007/BF01389192}},
}

\bib{CuntzQ}{article}{
  author={Cuntz, Joachim},
  title={$C^*$-algebras associated with the $ax+b$-semigroup over $\mathbb {N}$},
  conference={ title={$K$-theory and noncommutative geometry}, },
  book={ series={EMS Ser. Congr. Rep.}, publisher={Eur. Math. Soc., Z\"urich}, },
  date={2008},
  pages={201--215},
  doi={\href {http://dx.doi.org/10.4171/060-1/8}{10.4171/060-1/8}},
}

\bib{CLintegral2}{article}{
  author={Cuntz, Joachim},
  author={Li, Xin},
  title={{$C^*$}-algebras associated with integral domains and crossed products by actions on adele spaces},
  journal={J. Noncommut. Geom.},
  volume={5},
  year={2011},
  number={1},
  pages={1--37},
  issn={1661-6952},
  doi={\href {http://dx.doi.org/10.4171/JNCG/68}{10.4171/JNCG/68}},
}

\bib{CuntzVershik}{article}{
  author={Cuntz, Joachim},
  author={Vershik, Anatoly},
  title={{$C^*$}-algebras associated with endomorphisms and polymorphisms of compact abelian groups},
  year={2013},
  issn={0010-3616},
  journal={Comm. Math. Phys.},
  volume={321},
  number={1},
  doi={\href {http://dx.doi.org/10.1007/s00220-012-1647-0}{10.1007/s00220-012-1647-0}},
  publisher={Springer-Verlag},
  pages={157-179},
}

\bib{enders1}{article}{
  author={Enders, Dominic},
  title={Semiprojectivity for Kirchberg algebras},
  note={\href {http://arxiv.org/abs/1507.06091}{arxiv:1507.06091v1}},
}

\bib{Evans}{article}{
  author={Evans, D. Gwion},
  title={On the {$K$}-theory of higher rank graph {$C\sp *$}-algebras},
  journal={New York J. Math.},
  volume={14},
  year={2008},
  pages={1--31},
  issn={1076-9803},
  note={\href {http://nyjm.albany.edu:8000/j/2008/14_1.html}{nyjm:8000/j/2008/14\textunderscore 1}},
}

\bib{FS}{article}{
  author={Fowler, Neal J.},
  author={Sims, Aidan},
  title={Product systems over right-angled {A}rtin semigroups},
  journal={Trans. Amer. Math. Soc.},
  volume={354},
  year={2002},
  number={4},
  pages={1487--1509},
  issn={0002-9947},
  doi={\href {http://dx.doi.org/10.1090/S0002-9947-01-02911-7}{10.1090/S0002-9947-01-02911-7}},
}

\bib{Gl}{article}{
  author={Glimm, James G.},
  title={On a certain class of operator algebras},
  journal={Trans. Amer. Math. Soc.},
  volume={95},
  year={1960},
  pages={318--340},
}

\bib{hr}{book}{
  author={Hewitt, Edwin},
  author={Ross, Kenneth A.},
  title={Abstract harmonic analysis. Vol.~I},
  series={Grundlehren Math. Wiss.},
  volume={115},
  edition={2},
  note={Structure of topological groups, integration theory, group representations},
  publisher={Springer-Verlag, Berlin-New York},
  date={1979},
  pages={ix+519},
  isbn={3-540-09434-2},
}

\bib{Hir}{article}{
  author={Hirshberg, Ilan},
  title={On $C^*$-algebras associated to certain endomorphisms of discrete groups},
  journal={New York J. Math.},
  volume={8},
  year={2002},
  pages={99--109},
  issn={1076-9803},
  note={\href {http://nyjm.albany.edu:8000/j/2002/8_99.html}{nyjm.albany.edu:8000/j/2002/8\textunderscore 99}},
}

\bib{KOQ}{article}{
  author={Kaliszewski, Steve},
  author={Omland, Tron},
  author={Quigg, John},
  title={Cuntz-Li algebras from $a$-adic numbers},
  journal={Rev. Roumaine Math. Pures Appl.},
  volume={59},
  date={2014},
  number={3},
  pages={331--370},
  issn={0035-3965},
}

\bib{kasparov}{article}{
  author={Kasparov, Gennadi},
  title={Equivariant $KK$-theory and the Novikov conjecture},
  journal={Invent. Math.},
  volume={91},
  date={1988},
  number={1},
  pages={147--201},
  issn={0020-9910},
  doi={\href {http://dx.doi.org/10.1007/BF01404917}{10.1007/BF01404917}},
}

\bib{KatsuraIV}{article}{
  author={Katsura, Takeshi},
  title={A class of {$C^*$}-algebras generalizing both graph algebras and homeomorphism {$C^*$}-algebras. {IV}. {P}ure infiniteness},
  journal={J. Funct. Anal.},
  volume={254},
  year={2008},
  number={5},
  pages={1161--1187},
  issn={0022-1236},
  doi={\href {http://dx.doi.org/10.1016/j.jfa.2007.11.014}{10.1016/j.jfa.2007.11.014}},
}

\bib{Kir}{article}{
  author={Kirchberg, Eberhard},
  title={The classification of purely infinite C*-algebras using Kasparov's theory},
  journal={to appear in Fields Inst. Commun, Amer. Math. Soc., Providence, RI},
}

\bib{KP}{article}{
  author={Kumjian, Alex},
  author={Pask, David},
  title={Higher rank graph {$C^*$}-algebras},
  journal={New York J. Math.},
  volume={6},
  year={2000},
  pages={1--20},
  issn={1076-9803},
  note={\href {http://nyjm.albany.edu:8000/j/2000/6_1.html}{nyjm:8000/j/2000/6\textunderscore 1}},
}

\bib{Lac}{article}{
  author={Laca, Marcelo},
  title={From endomorphisms to automorphisms and back: dilations and full corners},
  journal={J. Lond. Math. Soc. (2)},
  volume={61},
  year={2000},
  number={3},
  pages={893--904},
  doi={\href {http://dx.doi.org/10.1112/S0024610799008492}{10.1112/S0024610799008492}},
}

\bib{LacRae}{article}{
  author={Laca, Marcelo},
  author={Raeburn, Iain},
  title={Semigroup crossed products and the {T}oeplitz algebras of nonabelian groups},
  journal={J. Funct. Anal.},
  volume={139},
  year={1996},
  number={2},
  pages={415--440},
  issn={0022-1236},
  doi={\href {http://dx.doi.org/10.1006/jfan.1996.0091}{10.1006/jfan.1996.0091}},
}

\bib{LarsenLi}{article}{
  author={Larsen, Nadia S.},
  author={Li, Xin},
  title={The $2$-adic ring {$C^*$}-algebra of the integers and its representations},
  journal={J. Funct. Anal.},
  volume={262},
  year={2012},
  number={4},
  pages={1392--1426},
  doi={\href {http://dx.doi.org/10.1016/j.jfa.2011.11.008}{10.1016/j.jfa.2011.11.008}},
}

\bib{Li1}{article}{
  author={Li, Xin},
  title={Semigroup {$C^*$}-algebras and amenability of semigroups},
  journal={J. Funct. Anal.},
  volume={262},
  year={2012},
  number={10},
  pages={4302--4340},
  issn={0022-1236},
  doi={\href {http://dx.doi.org/10.1016/j.jfa.2012.02.020}{10.1016/j.jfa.2012.02.020}},
}

\bib{LN2}{article}{
  author={Li, Xin},
  author={Norling, Magnus D.},
  title={Independent resolutions for totally disconnected dynamical systems {II}: {$C^*$}-algebraic case},
  journal={J. Operator Theory},
  volume={75},
  year={2016},
  number={1},
  pages={163--193},
  note={\href {http://www.theta.ro/jot/archive/2016-075-001/2016-075-001-009.pdf}{2016-075-001/2016-075-001-009}},
}

\bib{loring1}{book}{
  author={Loring, Terry A.},
  title={Lifting solutions to perturbing problems in $C^*$-algebras},
  series={Fields Inst. Monogr.},
  volume={8},
  publisher={Amer. Math. Soc., Providence, RI},
  date={1997},
  pages={x+165},
  isbn={0-8218-0602-5},
}

\bib{massey1}{article}{
  author={Massey, William S.},
  title={Exact couples in algebraic topology.~I, II},
  journal={Ann. of Math. (2)},
  volume={56},
  date={1952},
  pages={363--396},
  issn={0003-486X},
}

\bib{massey2}{article}{
  author={Massey, William S.},
  title={Exact couples in algebraic topology.~III, IV, V},
  journal={Ann. of Math. (2)},
  volume={57},
  date={1953},
  pages={248--286},
  issn={0003-486X},
}

\bib{Nic}{article}{
  author={Nica, Alexandru},
  title={$C^*$-algebras generated by isometries and {W}iener-{H}opf operators},
  journal={J. Operator Theory},
  volume={27},
  year={1992},
  number={1},
  pages={17--52},
  issn={0379-4024},
}

\bib{Oml}{article}{
  author={Omland, Tron},
  title={$C^*$-algebras associated with $a$-adic numbers},
  conference={ title={Operator algebra and dynamics}, },
  book={ series={Springer Proc. Math. Stat.}, volume={58}, publisher={Springer, Heidelberg}, },
  date={2013},
  pages={223--238},
  doi={\href {http://dx.doi.org/10.1007/978-3-642-39459-1_11}{10.1007/978-3-642-39459-1\textunderscore 11}},
}

\bib{Pas}{article}{
  author={Paschke, William L.},
  title={$K$-theory for actions of the circle group on $C^*$-algebras},
  journal={J. Operator Theory},
  volume={6},
  number={1},
  year={1981},
  pages={125--133},
}

\bib{Phi}{article}{
  author={Phillips, N. Christopher},
  title={A classification theorem for nuclear purely infinite simple $C^*$-algebras},
  journal={Doc. Math.},
  volume={5},
  year={2000},
  pages={49--114},
  issn={1431-0635},
}

\bib{PV}{article}{
  author={Pimsner, Mihai V.},
  author={Voiculescu, Dan-Virgil},
  title={Exact sequences for $K$-groups and Ext-groups of certain cross-product $C^*$-algebras},
  journal={J. Operator Theory},
  volume={4},
  number={1},
  year={1980},
  pages={93--118},
}

\bib{RS}{article}{
  author={Raeburn, Iain},
  author={Szyma{\'n}ski, Wojciech},
  title={Cuntz-Krieger algebras of infinite graphs and matrices},
  journal={Trans. Amer. Math. Soc.},
  volume={356},
  number={1},
  pages={39--59},
  year={2004},
  doi={\href {http://dx.doi.org/10.1090/S0002-9947-03-03341-5}{10.1090/S0002-9947-03-03341-5}},
}

\bib{rordam-zd}{article}{
  author={R{\o }rdam, Mikael},
  title={Classification of nuclear, simple $C^*$-algebras},
  conference={ title={Classification of nuclear $C^*$-algebras. Entropy in operator algebras} },
  book={ series={Encyclopaedia Math. Sci.}, volume={126}, publisher={Springer, Berlin}, },
  date={2002},
  pages={1--145},
  doi={\href {http://dx.doi.org/10.1007/978-3-662-04825-2_1}{10.1007/978-3-662-04825-2\textunderscore 1}},
}

\bib{savinienBellissard}{article}{
  author={Savinien, Jean},
  author={Bellissard, Jean},
  title={A spectral sequence for the $K$-theory of tiling spaces},
  journal={Ergodic Theory Dynam. Systems},
  volume={29},
  date={2009},
  number={3},
  pages={997--1031},
  issn={0143-3857},
  doi={\href {http://dx.doi.org/10.1017/S0143385708000539}{10.1017/S0143385708000539}},
}

\bib{schochet}{article}{
  author={Schochet, Claude L.},
  title={Topological methods for {$C^*$}-algebras.~I. Spectral sequences},
  year={1981},
  journal={Pacific J. Math.},
  volume={96},
  number={1},
  pages={193--211},
  note={\href {http://projecteuclid.org/euclid.pjm/1102734956}{euclid.pjm/1102734956}},
}

\bib{Sta1}{article}{
  author={Stammeier, Nicolai},
  title={On {$C^*$}-algebras of irreversible algebraic dynamical systems},
  journal={J. Funct. Anal.},
  volume={269},
  year={2015},
  number={4},
  pages={1136--1179},
  doi={\href {http://dx.doi.org/10.1016/j.jfa.2015.02.005}{10.1016/j.jfa.2015.02.005}},
}

\bib{Sta3}{article}{
  author={Stammeier, Nicolai},
  title={A boundary quotient diagram for right LCM semigroups},
  note={\href {http://arxiv.org/abs/1604.03172}{arxiv:1604.03172}},
}

\bib{Wil}{book}{
  author={Williams, Dana P.},
  title={Crossed products of $C^*$-algebras},
  series={Math. Surveys Monogr.},
  volume={134},
  publisher={Amer. Math. Soc., Providence, RI},
  date={2007},
  pages={xvi+528},
  isbn={978-0-8218-4242-3},
  isbn={0-8218-4242-0},
  doi={\href {http://dx.doi.org/10.1090/surv/134}{10.1090/surv/134}},
}

\bib{Z}{article}{
  author={Zhang, Shuang},
  title={Certain $C^*$-algebras with real rank zero and their corona and multiplier algebras.~I},
  journal={Pacific J. Math.},
  volume={155},
  number={1},
  year={1992},
  pages={169--197},
  note={\href {http://projecteuclid.org/euclid.pjm/1102635475}{euclid.pjm/1102635475}},
}

\end{biblist}
\end{document}